\documentclass[DIV=14]{scrartcl}
\usepackage[utf8]{inputenc}
\usepackage[colorlinks,urlcolor=Blue,citecolor=Green,linkcolor=BrickRed]{hyperref}

\usepackage[authoryear,round]{natbib}

\usepackage{amsmath,amssymb,amsfonts,amsthm, bm, mathtools}
\usepackage{stackengine}
\newcommand\ledot{~\dot{\le}~}
\usepackage{tabularx,booktabs,multirow}
\usepackage[algo2e,vlined,algoruled,linesnumbered]{algorithm2e}
\let\oldnl\nl
\newcommand{\nonl}{\renewcommand{\nl}{\let\nl\oldnl}}
\SetNlSty{}{}{}
\newcommand{\altline}[1]{\SetNlSty{}{\underline}{} \quad #1 \SetNlSty{}{}{}}
\DontPrintSemicolon
\usepackage{subdepth}
\usepackage{comment,color}
\usepackage{listings}
\usepackage{pgfplots}
\pgfplotsset{
    width=7cm, height=6.75cm,
    compat=1.18,
    table/search path={figures},
    grid=both,
    grid style={line width=.4pt, draw=gray!40},
    major grid style={line width=.4pt,draw=gray!70},
    every axis plot/.append style={line width=1.0pt, mark size = 2pt},
    legend style={font=\small},
    /tikz/every column/.append style={column sep=0.4cm},
    legend cell align={left}}

\usetikzlibrary{calc}
\usepackage{pdfcomment}
\newcommand\alttext[1]{
\useasboundingbox (current bounding box);
\draw let
  \p1 = (current bounding box.south west),
  \p2 = (current bounding box.north east)
  in node at (current bounding box.center)
    {\pdftooltip{\phantom{\rule{\dimexpr\x2-\x1\relax}{\dimexpr\y2-\y1\relax}}}{#1}};}

\usepackage[dvipsnames]{xcolor}
\colorlet{colorm1}{Turquoise!60}
\colorlet{colorm2}{Purple}
\colorlet{colorm3}{RoyalBlue}
\colorlet{colorm4}{Red}
\colorlet{colorm5}{Black}
\colorlet{colorm6}{Green}
\colorlet{colorm7}{Brown}

\pgfplotscreateplotcyclelist{methodlist}{
  color=colorm1, solid, every mark/.append style={solid}, mark=*\\%
  color=colorm2, densely dotted, every mark/.append style={solid}, mark=asterisk\\%
  color=colorm3, dashed,every mark/.append style={solid}, mark=square\\%
  color=colorm4, dash dot, every mark/.append style={solid}, mark=diamond\\%
  color=colorm5, dotted, every mark/.append style={solid}, mark={o}\\%
  color=colorm6, dotted, every mark/.append style={solid}, mark={+}\\%
    color=colorm7, dotted, every mark/.append style={solid}, mark={x}\\%
}

\usepackage{enumitem}
\usepackage{calc}
\setlist[enumerate,1]{labelindent=\dimexpr\the\leftmargini + 2pt\relax}
\setlist[enumerate,2]{labelindent=\dimexpr\the\leftmarginii + 2pt\relax}
\setlist[enumerate,3]{labelindent=\dimexpr\the\leftmarginiii + 2pt\relax}
\setlist[enumerate,4]{labelindent=\dimexpr\the\leftmarginiv + 2pt\relax}
\setlist[enumerate]{
    leftmargin=!,
    labelsep=1em,
    partopsep=0pt,
    itemsep=0pt,
    parsep=0pt,
    topsep=7pt
}
\setlist[itemize,1]{labelindent=\the\leftmargini}
\setlist[itemize,2]{labelindent=3mm}
\setlist[itemize,3]{labelindent=\the\leftmarginiii}
\setlist[itemize,4]{labelindent=\the\leftmarginiv}
\setlist[itemize]{
    leftmargin=!,
    labelsep=1em,
    partopsep=0pt,
    itemsep=0pt,
    parsep=0pt,
    topsep=7pt
}

\usepackage[dvipsnames]{xcolor}
\usepackage[capitalise]{cleveref}

\newcommand{\C}{\ensuremath{\mathbb C}}
\newcommand{\Cnn}{\ensuremath{\C^{n \times n}}}
\newcommand{\Cmn}{\ensuremath{\C^{m \times n}}}
\newcommand{\Cpq}{\ensuremath{\C^{p \times q}}}

\theoremstyle{thmstyletwo}
\newtheorem{theorem}{Theorem}
\newtheorem{lemma}[theorem]{Lemma}

\newtheorem{remark}[theorem]{Remark}
\theoremstyle{definition}

\numberwithin{equation}{section}

\DeclareMathOperator{\diag}{diag}
\DeclareMathOperator{\opvec}{vec}
\DeclareMathOperator{\fl}{fl}
\newcommand{\wh}{\widehat}
\newcommand{\wt}{\widetilde}
\newcommand{\wc}{\check}
\renewcommand{\Delta}{\varDelta}

\newcommand{\deq}{\ensuremath{:=}}
\newcommand{\deqr}{\ensuremath{=:}}
\DeclareMathOperator{\cond}{\mu}

\renewcommand{\Lambda}{\varLambda}
\renewcommand{\Sigma}{\varSigma}
\renewcommand{\Gamma}{\varGamma}

\newcommand{\algnaive}{\texttt{naive}}
\newcommand{\algss}{\texttt{sqrt\_schur}}
\newcommand{\algsc}{\texttt{chol\_schur}}
\newcommand{\algscp}{\texttt{chol\_schur\_pd}}

\newcommand{\algscsc}{\texttt{schur\_schur}}

\newcommand{\D}[2]{D#1(#2)}
\newcommand{\Df}[3]{D#1(#2)[#3]}

\title{Computing matrix functions associated with a Hermitian--definite pencil}

\author{
  Dario A. Bini
  \footnote{Department of Mathematics, University of Pisa, Largo B.
  Pontecorvo, 5, 56127, Pisa, Italy.}
  \and
  Massimiliano Fasi
  \footnote{School of Computer Science, University of Leeds, Woodhouse Lane, Leeds
  LS2 9JT, UK.}
  \and
  Bruno Iannazzo
  \footnote{Department of Mathematics and Computer Science, University of
  Perugia, Italy.}}

\date{}

\colorlet{revcol}{black}
\newcommand{\hilight}{\color{revcol}}

\begin{document}

\maketitle

\abstract{We consider the numerical evaluation of the quantity $Af(A^{-1}B)$, where $A$ is Hermitian positive definite, $B$ is Hermitian, and $f$ is a function defined on the spectrum of $A^{-1}B$. This problem is related to the Hermitian--definite matrix pencil $B-\lambda A$. We study the conditioning of the problem, and we introduce several algorithms that combine the Schur decomposition with either the matrix square root or the Cholesky factorization. We study the numerical behavior of these algorithms in floating-point arithmetic, assess their computational costs, and compare their numerical performance. Our analysis suggests that the algorithms based on the Cholesky factorization will be more accurate and efficient than those based on the matrix square root. This is confirmed by our numerical experiments.}


\bigskip
\begin{center}
\textbf{Dedicated to the memory of Nicholas J. Higham}
\end{center}

\section{Introduction}
\label{sec:intro}

A function $f:\mathcal I \subset \mathbb R \rightarrow \C$ can be extended to diagonalizable matrices with eigenvalues in $\mathcal I$ as a primary matrix function.
The concept is thoroughly discussed in the monograph by \citet{higham:book08}.

{\color{revcol} Let $A \in \Cnn$ be Hermitian positive definite, let $B \in \Cnn$ be Hermitian, and let $A^{1/2}$ denote the principal square root of $A$. Since $A^{-1}B=A^{-1/2}(A^{-1/2}BA^{-1/2})A^{1/2}$,  the matrix $A^{-1}B$ is similar to the Hermitian matrix $A^{-1/2}BA^{-1/2}$ and is therefore diagonalizable with real eigenvalues. One} can leverage the concept of primary matrix function to give meaning to the expression
\begin{equation}\label{eq:1}
    \varphi(A,B) \deq{} Af(A^{-1}B),
\end{equation}
where $f$ is defined on the spectrum of $A^{-1}B$.
It is often convenient to consider the formulation
\begin{equation}\label{eq:2}
\varphi(A,B) = A^{1/2}f(A^{-1/2}BA^{-1/2})A^{1/2},
\end{equation}
where $A^{1/2}$ is the principal square root of $A$, which is equivalent to~\eqref{eq:1} but has a symmetric form.

Here, we consider the problem of computing $\varphi(A,B)$ as efficiently as possible, taking into account both computational cost and numerical stability.

Expressions like \eqref{eq:1} and \eqref{eq:2} appear frequently in the applied mathematics literature. \citet{kuan80} use the formulation~\eqref{eq:2}, for $B$ positive definite, to define means of positive definite operators (and in particular matrices) in a systematic way.
Notable cases are the arithmetic, harmonic, and geometric means, which are obtained for
\[
    f(x)=\frac{1+x}{2},\quad
    f(x)=\frac{2x}{1+x},\quad\text{and}\quad
    f(x)=\sqrt{x},
\]
respectively.
Setting $f(x)=x^t$, for $t\in[0,1]$, yields the weighted geometric mean
$A\#_t B \deq A(A^{-1}B)^t$,
and the weighted power mean of two matrices is obtained with $f(x)=((1-t)+tx^p)^{1/p}$ for $p\ne 0$. We refer the reader to the recent survey by \citet{bi24} for a discussion of geometric means of two matrices.

\citet[sect.~2]{parl76} considers $A$ and $B$ with no structure and discusses methods to compute $f(A^{-1}B)$ without forming $A^{-1}$. The relevant application is the system of ordinary differential equations
\[
    Au'(t)=Bu(t),\qquad u(0)=u_0,
\]
whose explicit solution is $u(t) = \exp(A^{-1}Bt)u_0$. Similar functions arise in the numerical solution of differential-algebraic equations~\citep{book}.

Functions like $f(A^{-1}B)$ or $\varphi(A,B)$ are related to the matrix pencil $B-\lambda A$, which appears in the generalized eigenvalue problem $(B-\lambda A)x=0$. When $A$ is nonsingular, the generalized eigenvalue problem is equivalent to the standard eigenvalue problem $(A^{-1}B-\lambda I)x=0$. For this reason, \citet{parl76} refers to \(f(A^{-1}B)\) as a \emph{function of the pencil} $B-\lambda A$. {\color{revcol}In our case, we use the symmetric form $Af(A^{-1}B)$ and the pencil $B-\lambda A$ is {\em Hermitian--definite}, in the sense that $B$ is Hermitian and $A$ is positive definite.}

Two pencils are spectrally equivalent if they have the same sets of eigenvalues.
The pencil $\widetilde B-\lambda\widetilde A$ is spectrally equivalent to $B-\lambda A$ if there exist two nonsingular matrices $P$ and $Q$ such that $\widetilde B=PBQ$ and $\widetilde A= PAQ$.
The function $\varphi$ is compatible with pencil equivalence, as shown by property~\ref{it:p2} of the following \cref{thm:4}, which summarizes some of the properties of~$\varphi$.

\begin{theorem}\label{thm:4}
Let $A \in \Cnn$ be a positive definite Hermitian matrix, let $B \in \Cnn$ be Hermitian, let $\varphi(A,B) \deq{} Af(A^{-1}B)$, with $f:\mathcal I\rightarrow \C$, $\mathcal I\subset \mathbb R$, and let $\lambda_1,\ldots,\lambda_n\in\mathcal I$ be the eigenvalues of the pencil $B-\lambda A$.
Then, we have:
\begin{enumerate}[label=(\alph*), ref=(\alph*)]
\item $Af(A^{-1}B) = A^{1/2}f(A^{-1/2}BA^{-1/2})A^{1/2} = f(BA^{-1})A$ (symmetry);\label{it:p1}
\item $P\varphi(A,B)Q=\varphi(PAQ,PBQ)$, for $P,Q\in\C^{n\times n}$ nonsingular matrices (commutativity with pencil equivalence); \label{it:p2}
\item $\varphi(\alpha A,\alpha B)=\alpha\varphi(A,B)$, for $\alpha\in\mathbb C\setminus\{0\}$ (homogeneity);\label{it:p3}
\item a nonsingular $M$ exists such that $M^*\varphi(A,B)M=\varphi(I,D)=f(D)$, where $D=\diag(\lambda_1,\ldots,\lambda_n)$;\label{it:p4}
\item $\varphi(A,B)=Ap(A^{-1}B)$, where $p$ is the polynomial interpolating $f$ at $\lambda_1,\ldots,\lambda_n$; in particular, if $n=2$, then
{\hilight
\begin{equation}\label{eq:15}
\varphi(A,B) = \begin{dcases}
\frac{f(\lambda_2)\lambda_1-f(\lambda_1)\lambda_2}{\lambda_1-\lambda_2}A
+\frac{f(\lambda_1)-f(\lambda_2)}{\lambda_1-\lambda_2} B,\quad&\lambda_1 \neq \lambda_2,\\
f(\lambda_1)A,&\lambda_1 = \lambda_2;
\end{dcases}
\end{equation}
}.\label{it:p5}

\item if $A$ and $B$ are block diagonal matrices with the same block structure, then $\varphi(A,B)$ is block diagonal with the same structure;\label{it:p6}
\item if $f(\mathcal I)\subset \mathbb R$ then $\varphi(A,B)$ is Hermitian; if $f(\mathcal I)\subset (0,\infty)$ then $\varphi(A,B)$ is positive definite.
\end{enumerate}
\end{theorem}
\begin{proof}
{\color{revcol}\Cref{it:p1,it:p2} follow from the commutativity of primary matrix functions with similarities: if $f$ is defined on the spectrum of $A$ and $M$ is nonsingular, then $f(M^{-1}AM)=M^{-1}f(A)M$. \Cref{it:p3} is straightforward. {\hilight \Cref{it:p4} follows by setting $M=A^{-1/2}U$, where $U^*(A^{-1/2}BA^{-1/2})U=D$. Indeed, $M^*AM=I$ and $M^*BM=D$, thus \cref{it:p2} with $P=M^*$ and $Q=M$ gives {$M^*\varphi(A,B)M=\varphi(I,D)=f(D)$.}}
\Cref{it:p5,it:p6} follow from the analogous properties of primary matrix functions~\cite[Chap.~1]{higham:book08}.}
{\hilight The last item follows from the proof of \cref{it:p4}, since $\varphi(A,B)=A^{1/2}Uf(D)U^*A^{1/2}$, which is Hermitian if $f(D)$ is real, and positive definite if the diagonal of $f(D) = \diag\bigl(f(\lambda_1),\ldots,f(\lambda_n)\bigr)$ is positive.}
\end{proof}

If $B$ is nonsingular, swapping $A$ and $B$ is often useful in the analysis of~\eqref{eq:1} and~\eqref{eq:2}. For this purpose, it is convenient to introduce the dual function $\widehat f(x) \deq{} xf(x^{-1})$.
The dual of $\widehat f$ is $f$, and if $f$ is defined on the spectrum of $A^{-1}B$, then $\widehat f$ is defined on the spectrum of $B^{-1}A$ and $Af(A^{-1}B)=B\widehat f(B^{-1}A)$.
\Cref{tab:1} lists some examples of functions together with their duals.

\begin{table}[t]
\centering
\caption{Examples of dual functions,
with $t\in[0,1]$ and $p\ne 0$.}\label{tab:1}
\begin{tabularx}{\textwidth}
{X
>{\centering\arraybackslash}p{3cm}
>{\centering\arraybackslash}p{0.8cm}
>{\centering\arraybackslash}p{3cm}
>{\centering\arraybackslash}p{1.5cm}
>{\centering\arraybackslash}p{1.5cm}
}
\toprule
Expression for $f$ & $((1-t)+tx)/2$ & $x^t$ & $((1-t)+tx^p)^{1/p}$ & $\exp(x)$ & $\log(x)$ \\
Corresponding $\widehat f$ & $((1-t)x+t)/2$ & $x^{1-t}$ & $((1-t)x^p+t)^{1/p}$ & $x\exp(x^{-1})$ & $-x\log(x)$ \\
\bottomrule
\end{tabularx}
\end{table}

The expression \eqref{eq:1} gives rise to two computational problems:
\begin{enumerate}[label=P\arabic*., ref=P\arabic*, leftmargin=23pt]
\item computing $Af(A^{-1}B)$, where $A$ and $B$ are of moderate size; and\label{it:a}
\item computing $Af(A^{-1}B)v$, where $A$ and $B$ are large and sparse, and $v$ is a vector.\label{it:b}
\end{enumerate}
It is highly desirable to avoid, in~\ref{it:a}, the computation of $A^{-1}B$, and, in~\ref{it:b}, the computation of $f(A^{-1}B)$, which is dense, in general, even when $A$ and $B$ are sparse.
Our analysis will focus on \ref{it:a}.

The aforementioned Kubo--Ando means of two matrices are a straightforward application of \ref{it:a}.
Computing the weighted geometric mean (Karcher mean) of three or more matrices~\citep{bi13,ip18,q03}, or its derivative \citep{1}, requires the evaluation of $A\exp(A^{-1}B)$ and $A\log(A^{-1}B)$, both of which are of the form~\eqref{eq:1}. Functions of the type $Af(A^{-1}B)$ arise in the definition of the matrix geometric mean~\citep{q01,tsm21,q04,q05}, used in recent literature to illustrate the behavior of Riemannian optimization algorithms.

Applications of~\ref{it:b} can be found in domain decomposition methods, both in established~\citep{al1,al2, alk} and in more recent~\citep{d1,hkmw25} work, and in techniques to cluster signed networks~\citep{fi,mth1}, where one needs to compute the action of the weighted geometric mean of two matrices on a vector.

The related problem $f(A^{-1}B)v$ appears in finite element methods for reaction--diffusion equations: for example, \citet{bak} consider the case where $f$ is a fractional power, while \citet{ytfi,ytml} consider $f(z) = E_{\gamma,1}(a_1 z^{\alpha/2})$ or \mbox{$f(z)=\bigl(1+a_0 z^{\alpha/2}\bigr)^{-1}$}.
Here, $E_{\gamma,1}$ is the \mbox{Mittag-Leffler} function and $a_0$ and $a_1$ are constants.
In these problems, using a Krylov subspace method typically yields the computation of $Ag(A^{-1}B)$, for some function $g$, where $A$ and $B$ have small or moderate size.

After revising some background material in \cref{sec:preliminaries}, in \cref{sec:conditioning} we discuss the conditioning of the problem.
In \cref{sec:chsh-alg}, we describe several numerical algorithms for the evaluation of~\eqref{eq:1}, studying their numerical stability in floating-point arithmetic.
We compare these algorithms, in terms of accuracy and performance, in \cref{sec:num-exp}.
\Cref{sec:conclusions} summarizes possible directions for future work.

\section{Preliminaries}
\label{sec:preliminaries}

Given $A \in \Cmn$ and $B \in \Cpq$,
we denote by {\color{revcol}$A^T$ and $A^*$ the transpose and conjugate transpose of $A$, respectively, by} $\opvec(A)$ the vector of length $mn$ that stacks the columns of $A$ from left to right, and by $A \otimes B$ the Kronecker product of $A$ and $B$. We will use the property $\opvec(AXB)=(B^T\otimes A)\opvec(X)$, which holds whenever the product $AXB$ is defined.

For a complex function $f$, we define the divided differences
\[
    f[\lambda_i,\lambda_j]\deq{}
    \begin{cases}
    \dfrac{f(\lambda_j)-f(\lambda_i)}{\lambda_j-\lambda_i}, & \mbox{ if }\lambda_i\ne \lambda_j,\\
    f'(\lambda_i), & \mbox{ if }\lambda_i=\lambda_j.
    \end{cases}
\]
{\hilight With this notation, equation \eqref{eq:15} for $n=2$ can be written as $Af(A^{-1}B) = \wh f[1/\lambda_1,1/\lambda_2]A+f[\lambda_1,\lambda_2]B$, as long as $0\ne \lambda_1\ne \lambda_2\ne 0$.}

A linear map $L:\Cmn \rightarrow \Cpq$ can be represented by a matrix $M\in\mathbb C^{(pq)\times (mn)}$ such that $\opvec(L[H])=M\opvec(H)$ for any $H\in\Cmn$. The matrix
$M$ is the Kronecker form of the map $L$ and represents $L$ in the vec basis.
This form will be used for the Fréchet derivatives of primary matrix functions.

A primary matrix function obtained from the complex function $f$ is continuously differentiable at $A\in\Cnn$ if it is sufficiently smooth, for instance analytic, at the eigenvalues of $A$; see \citep[Thm.~3.8]{higham:book08} for the precise statement. We denote by $\D{f}{A} : \Cnn \rightarrow \Cnn$ the Fréchet derivative of $f$ at $A$. Thus $\Df{f}{A}{H}$ is the directional derivative of $f$ at $A$ in the direction $H \in \Cnn$, namely
\begin{equation}\label{eq:10}
\Df{f}{A}{H} = \lim_{\varepsilon\rightarrow 0} \frac{1}{\varepsilon}\bigl(f(A+\varepsilon H)-f(A)\bigr).
\end{equation}
Equation~\eqref{eq:10} implies the commutativity with similarities, that is,
\begin{equation}\label{eq:8}
    M^{-1}\Df{f}{A}{H}M=\Df{f}{M^{-1}AM}{M^{-1}HM},
\end{equation}
for $f$ differentiable at $A$ and $M$ nonsingular. We also recall the Daleckiĭ--Kreĭn theorem \cite[Thm.~3.11]{higham:book08},
\begin{equation}\label{eq:9}
    \Df{f}{\diag(\lambda_1,\ldots,\lambda_n)}{H} = F\circ H,\qquad F_{ij}=f[\lambda_i,\lambda_j],
\end{equation}
where $A\circ B \in \Cmn$ denotes the component-wise (Hadamard) product of $A, B \in \Cmn$.

Throughout this work, we denote the Frobenius norm and the 2-norm by $\| \cdot \|_F$ and $\| \cdot \|$, respectively, and we use the 2-norm condition number $\cond(A) = \|A\| \|A^{-1}\|$.
The unit roundoff of the working floating-point arithmetic will be denoted by $u$, and dotted relations only hold to the first order, ignoring higher terms in $u$. For example, $\doteq$ and $\ledot$ denote first-order equivalence and inequality, respectively.

\section{Conditioning}
\label{sec:conditioning}

If $f$ is differentiable in a neighborhood of $A^{-1}B$, then the function $\varphi(A,B)=Af(A^{-1}B)$ is differentiable in a neighborhood of $(A,B)$. We can measure the sensitivity of $\varphi(A,B)$ with respect to perturbations using the relative Frobenius-norm condition number \cite{rice} of $\varphi(A,B)$, which is
\begin{equation}\label{eq:4}
    \hbox{cond}(\varphi;A,B)\deq{}\frac{\|\D{\varphi}{A,B}\|_F\|[A\ B]\|_F}{\|\varphi(A,B)\|_F},
\end{equation}
where $\D{\varphi}{A,B}$ is the Fréchet derivative of $\varphi$ at $(A,B)$ and
\[
    \|D\varphi(A,B)\|_F\deq{}\max_{[H\ L]\ne 0}\frac{\|\Df{\varphi}{A,B}{H,L}\|_F}{\|[H\ L]\|_F}.
\]

The following result provides an explicit expression for the Fréchet derivative of $\varphi$, whereby an explicit expression for the condition number can be obtained.
This will allow us to derive useful bounds.

\begin{lemma}\label[lemma]{thm:1}
Let $A,B\in\Cnn$, with $A$ nonsingular, and let $f$ be a function differentiable at $A^{-1}B$. Then the Fréchet derivative of $\varphi(A,B) \deq Af(A^{-1}B)$ in the direction $[H, L]$, with $H,L\in\Cnn$, is
\begin{equation}\label{eq:5}
    \Df{\varphi}{A,B}{H,L} = Hf(A^{-1}B)+A\Df{f}{A^{-1}B}{-A^{-1}HA^{-1}B+A^{-1}L}.
\end{equation}
If $B$ is nonsingular, then the function $\widehat f(x)\deq{}xf(x^{-1})$ is differentiable at $B^{-1}A$ and
\begin{equation}\label{eq:6}
    \Df{\varphi}{A,B}{H,L} = B\Df{\widehat f}{B^{-1}A}{B^{-1}H}+A\Df{f}{A^{-1}B}{A^{-1}L}.
\end{equation}
\end{lemma}
\begin{proof} Let $\zeta(A,B)\deq{}A^{-1}B$. Using the chain rule and the product rule yields
\begin{equation}\label{eq:3}
\begin{split}
    \Df{\varphi}{A,B}{H,L}
    & =
    Hf(A^{-1}B)+A\Df{f}{A^{-1}B}{\Df{\zeta}{A,B}{H,L}}\\
    & =
    Hf(A^{-1}B)+A\Df{f}{A^{-1}B}{-A^{-1}HA^{-1}B+A^{-1}L}.
\end{split}
\end{equation}
If $B$ is nonsingular, then the eigenvalues of $B^{-1}A$ are the reciprocals of those of $A^{-1}B$. Thus, $\widehat f$~is differentiable at $B^{-1}A$, and applying~\eqref{eq:3} to the function $\psi(B,A) \deq{} B\wh f(B^{-1}A)=\varphi(A,B)$,
we obtain
\begin{equation}\label{eq:7}
    \Df{\varphi}{A,B}{H,L} = \Df{\psi}{B,A}{L,H} = L\widehat f(B^{-1}A)+
    B\Df{\widehat f}{B^{-1}A}{-B^{-1}LB^{-1}A+B^{-1}H}.
\end{equation}
Finally,
\[
  \begin{aligned}
    \Df{\varphi}{A,B}{H,L}
    &= \Df{\varphi}{A,B}{H,0}+\Df{\varphi}{A,B}{0,L}\\
    &= B\Df{\widehat f}{B^{-1}A}{B^{-1}H} + A\Df{f}{A^{-1}B}{A^{-1}L},
  \end{aligned}
\]
where we have used \eqref{eq:7} for the first summand and \eqref{eq:3} for the second.
\end{proof}

If we work in the vec basis and assume that $A$ and $B$ are positive definite, then  a more explicit expression can be given for the derivative of $\varphi$.

\begin{lemma}\label[lemma]{thm:2}
In the notation of \cref{thm:1}, let $A$ and $B$ be positive definite, and let $Z_1$ and $Z_2$ be such that $Z_1^{-1}A^{-1}BZ_1=\diag(\lambda_1,\ldots,\lambda_n)\deqr{}\Lambda$ and $Z_2^{-1}B^{-1}AZ_2=\diag(1/\lambda_1,\ldots,1/\lambda_n)=\Lambda^{-1}$. Then,
\[
    \opvec(\Df{\varphi}{A,B}{H,L})=M_1\opvec(L)+M_2\opvec(H),
\]
where $M_p=(Z_p^{-T}\otimes AZ_p)\diag(\opvec(F^{(p)}))(Z_p^T\otimes (AZ_p)^{-1})$, for $p=1,2$, with $F^{(1)}_{ij}=f[\lambda_i,\lambda_j]$ and \mbox{$F^{(2)}_{ij} = \wh f[1/\lambda_i,1/\lambda_j]$}, for $i,j=1,\ldots,n$.
In particular, with
\begin{align*}
\wh Z_1&=A^{-1/2}Q_1,\qquad Q_1^*A^{-1/2}BA^{-1/2}Q_1=\Lambda,\\
\wh Z_2&=B^{-1/2}Q_2,\qquad Q_2^*B^{-1/2}AB^{-1/2}Q_2=\Lambda^{-1},
\end{align*}
where $Q_1$ and $Q_2$ are unitary, we obtain
\begin{equation}\label{eq:13}
    M_p = (\wh Z_p^{-T}\otimes \wh Z_p^{-*})\diag(\opvec(F^{(p)}))
    (\wh Z_p^T \otimes \wh Z_p^*),\qquad p=1,2.
\end{equation}
\end{lemma}

\begin{proof}
The idea of the proof is to take the vec of \eqref{eq:6} and manipulate its right-hand side, starting from $A\Df{f}{A^{-1}B}{A^{-1}L}$.
By \eqref{eq:8}, we get
\begin{equation}\label{eq:11}
  \begin{aligned}
    A\Df{f}{A^{-1}B}{A^{-1}L}
    &= AZ_1\Df{f}{Z_1^{-1}A^{-1}BZ_1}{Z_1^{-1}A^{-1}LZ_1}Z_1^{-1}\\
    &= AZ_1\Df{f}{\Lambda}{(AZ_1)^{-1}LZ_1}Z_1^{-1},
  \end{aligned}
\end{equation}
and the Daleckiĭ--Kreĭn theorem \eqref{eq:9} implies that
\[
    \Df{f}{\Lambda}{(AZ_1)^{-1}LZ_1}
    = F^{(1)}
    \circ \bigl((AZ_1)^{-1}LZ_1\bigr),
\]
Using the properties of the vec operator, we can write
\begin{equation}\label{eq:12}
\begin{split}
\opvec(A\Df{f}{A^{-1}B}{A^{-1}L})
& = (Z_1^{-T}\otimes AZ_1)\opvec\bigl(F^{(1)}\circ((AZ_1)^{-1}LZ_1)\bigr)\\
& = (Z_1^{-T}\otimes AZ_1)\diag(\opvec(F^{(1)}))(Z_1^T\otimes (AZ_1)^{-1})\opvec(L)\\
&  = M_1 \opvec(L).
\end{split}
\end{equation}
Repeating the argument for $\opvec\bigl(B\Df{\widehat f}{B^{-1}A}{B^{-1}H}\bigr)$ and using \cref{thm:1} completes the first part of the proof.
For the second part, we use the first part and observe that, if $\wh Z_1\deq{}A^{-1/2}Q_1$, then $\wh Z_1$ diagonalizes $A^{-1}B$ and $A\wh Z_1=\wh Z_1^{-*}$; and if $\wh Z_2\deq{}B^{-1/2}Q_2$, then $\wh Z_2$ diagonalizes $B^{-1}A$ and $B\wh Z_2=\wh Z_2^{-*}$.
\end{proof}

In order to bound the condition number, we need a technical lemma.
\begin{lemma}\label[lemma]{thm:3}
In the notation of \cref{thm:2}, the following bounds hold
\[
\|M_1\|\le \mu(A)\max_{i,j=1,\ldots,n}|f[\lambda_i,\lambda_j]|,\qquad \|M_2\|\le \mu(B)\max_{i,j=1,\ldots,n}|\widehat f[1/\lambda_i,1/\lambda_j]|.
\]
\end{lemma}
\begin{proof}
Using equation \eqref{eq:13}, we obtain
$
    \|M_p\|\le \mu(\wh Z_p^T \otimes \wh Z_p^*)\|\diag(\opvec(F^{(p)}))
    \|,
$
for $p=1,2$.
Therefore,
\[
    \|\wh Z_1^T\otimes \wh Z_1^*\|=\|\wh Z_1^T\|\|\wh Z_1^*\|=\|\wh Z_1\|^2=
    {\color{revcol} \|A^{-1/2}\|^2 =\|A^{-1}\|},
\]
and analogous identities hold for $\|\wh Z_1^{-T}\otimes \wh Z_1^{-*}\|$, $\|\wh Z_2^T\otimes \wh Z_2^*\|$, and $\|\wh Z_2^{-T}\otimes \wh Z_2^{-*}\|$. Thus,
$
\mu(\wh Z_1^T\otimes \wh Z_1^*)=\mu(A)
$
and
$
\mu(\wh Z_2^T\otimes \wh Z_2^*)=\mu(B).
$
Noting that the 2-norm of a diagonal matrix is the maximum modulus of its diagonal elements concludes the proof.
\end{proof}

Using \cref{thm:3}, we can easily bound the Frobenius norm of the Fréchet derivative of $\varphi$, since
\[
\begin{split}
    \|D\varphi(A,B)\|_F & = \|[M_2\ M_1]\|\le\sqrt{\|M_1\|^2+\|M_2\|^2}\\
    & \le \sqrt{\mu(A)^2\max_{i,j=1,\ldots,n}|f[\lambda_i,\lambda_j]|^2
    +\mu(B)^2\max_{i,j=1,\ldots,n}|\wh f[1/\lambda_i,1/\lambda_j]|^2}.
    \end{split}
\]
Similar bounds involving only $\mu(A)$ or $\mu(B)$ can be obtained---an example is
\begin{equation}\label{eq:14}
    \|\D{\varphi}{A,B}\|_F \le
    \mu(A)\sqrt{\max_{i,j=1,\ldots,n }|f[\lambda_i,\lambda_j]|^2+\max_{i,j=1,\ldots,n} |f(\lambda_j)-f[\lambda_i,\lambda_j]\lambda_j|^2}.
\end{equation}
If $f$ and $\wh f$ are convex and monotone, then the bounds simplify, since
\begin{equation*}
\max |f[\lambda_i,\lambda_j]|=f'(\lambda_M)
\text{ and }
\max |\wh f[1/\lambda_i,1/\lambda_j]|={\wh f}'(1/\lambda_m),
\end{equation*}
where $\lambda_m$ and $\lambda_M$ are the smallest and the largest eigenvalues of $A^{-1}B$, respectively.
For instance, using \eqref{eq:14} with $f(x)=\sqrt{x}$ yields the bound \cite[eq.~8]{bi24} on the condition number of the matrix geometric mean.

\section{The Cholesky--Schur algorithm}\label{sec:chsh-alg}
\begin{algorithm2e}[t]
\caption{Square root-based algorithm for computing $Af(A^{-1}B)$.}
\label{alg:1}
\KwIn{A positive definite matrix $A$, a Hermitian matrix $B$, and a real-valued function $f$.}
\KwOut{The Hermitian matrix $S_5 = Af(A^{-1}B)$.}
$S_1\deq A^{1/2}$\;\label[line]{ln:ss-1}
$S_2\deq S_1^{-1}$\;\label[line]{ln:ss-2}
$S_3\deq S_2BS_2^*$\;\label{ln:ss-3}
\altline{\emph{Faster variant:} Replace Lines \ref{ln:ss-2} and \ref{ln:ss-3} with $S_1\deqr R^*R$; $S_3=\smash{R^{-1}(R^{-*}BR^{-1})R^{-*}}$.\;} \label{ln:ss-3a}
$S_3\deqr Q\Lambda Q^*$\quad (Schur decomposition)\;\label{ln:ss-4}
$S_4\deq Qf(\Lambda )Q^*$\;\label{ln:ss-5}
$S_5\deq S_1S_4S_1^*$\;\label{ln:ss-6}
\Return{$S_5$}\;\label{ln:ss-7}
\end{algorithm2e}

A naive approach to compute $F=Af(A^{-1}B)$ exploits~\eqref{eq:2}.
The pseudocode is given in \Cref{alg:1}. Besides the straightforward version, the pseudocode reports
a variant of the algorithm---here and below, alternative lines are indicated with underlined line numbers. This procedure requires the computation of a matrix square root, a matrix inversion, five matrix multiplications, three of which produce a
Hermitian matrix, and the Schur factorization of a Hermitian matrix.
Observe that $S_1$ and $S_2$ in \cref{alg:1} are Hermitian and $S_1^*$ and $S_2^*$ could be replaced by $S_1$ and $S_2$, respectively.

We can use an asymptotic analysis to estimate the number of floating-point operations (flops) required by \cref{alg:1}. \Cref{tab:cost} in the Appendix reports the dominant $n^3$ term in the flop count of the matrix operations used.
Computing $A^{1/2}$ as $U \Lambda ^{1/2} U^*$, where $A=U \Lambda U^*$ is a Schur decomposition of $A$, requires $10n^3$ flops. Inverting the Hermitian $S_1$ on Line \ref{ln:ss-2} requires $n^3$ flops.
Lines \ref{ln:ss-3} and \ref{ln:ss-6} are matrix multiplications with Hermitian results and require $2n^3+n^3$ flops each. The Schur decomposition on Line \ref{ln:ss-4} requires $9n^3$ flops, and the cost of Line \ref{ln:ss-5} is dominated by a matrix multiplication producing a Hermitian result, which costs $n^3$ flops. Thus \cref{alg:1} asymptotically requires $27n^3$ flops.

This cost can be reduced by using the Cholesky factorization $S_1 \deqr{} R^*R$, where $R$ is an upper triangular matrix with positive diagonal entries.
Computing $S_3$ as $R^{-1}KR^{-*}$, where $K=R^{-*}BR^{-1}$, only requires $3n^3$ flops once $S_1$ is available, thus this variant of \cref{alg:1} requires $26n^3$ flops in total.

A further reduction can be obtained  if $S_5$ is computed as $S_tf(\Lambda) S_t^*$, where $S_t = S_1 Q$, without forming $S_4$. In this case, the overall cost is reduced to $25n^3$ flops.
Compared with \cref{alg:2} below, however, these modifications have only a minor effect on the computational cost of the algorithm.

{\color{revcol} \Cref{alg:ref} describes an alternative approach.
This version requires two Schur decompositions and four matrix multiplications, two of which (on Line \ref{ln:r-4} and
\ref{ln:r-7}) produce a Hermitian result. From \cref{tab:cost}, we get an overall cost of $24n^3$. }

{ \begin{algorithm2e}[t] \color{revcol}
\caption{Algorithm for computing $Af(A^{-1}B)$ using two Schur decompositions.}\label{alg:ref}

\KwIn{A positive definite matrix $A$, a Hermitian matrix $B$, and a real-valued function $f$.}
\KwOut{The Hermitian matrix $S_5 = Af(A^{-1}B)$.}
$A\deqr U\Lambda U^*$ \quad (Schur decomposition) \;\label{ln:r-1}
$S_1 = \Lambda^{1/2}$\;\label{ln:r-2}
$S_2=US_1^{-1}$\;\label{ln:r-3}
$S_3\deq S_2^*B S_2$\;\label{ln:r-4}
$S_3\deqr V\Lambda_1 V^*$\quad (Schur decomposition)\;\label{ln:r-5}
$S_4\deq U{\hilight S_1}V$\;\label{ln:r-6}
$S_5=S_4 f(\Lambda_1)S_4^*$\;\label{ln:r-7}
\Return{$S_5$}\;
\end{algorithm2e}
}

A more effective approach relies on the Cholesky factorization $A \deqr{} R_A^*R_A$.
This yields the Cholesky--Schur method in \cref{alg:2}.
The basic version of this algorithm is similar to \cref{alg:1} and
comprises Lines \ref{ln:sc-1}--\ref{ln:sc-3}, \ref{ln:sc-4}--\ref{ln:sc-6}, and \ref{ln:sc-8}. The first variant replaces Lines \ref{ln:sc-2} and \ref{ln:sc-3} by the faster Line \ref{ln:sc-3a}, which avoids inverting $R_A^*$ and requires the solution of two triangular systems with $n$ right-hand sides, one of which produces a Hermitian matrix. The second variant replaces Lines \ref{ln:sc-5} and \ref{ln:sc-6} by Line \ref{ln:sc-7}.

\begin{algorithm2e}[t]
\caption{Cholesky--Schur algorithm for computing $Af(A^{-1}B).$}
\label{alg:2}
\KwIn{A positive definite matrix $A$, a Hermitian matrix $B$, and a real-valued function $f$.}
\KwOut{The Hermitian matrix $S_5 = Af(A^{-1}B)$.}
$S_1\deq R_A^*$, where $A=:R_A^*R_A$ is the Cholesky decomposition of $A$.\;\label{ln:sc-1}
$S_2\deq S_1^{-1}$\; \label{ln:sc-2}
$S_3\deq S_2 B S_2^{*}$\; \label{ln:sc-3}
\altline{\emph{Faster variant:} Replace Lines \ref{ln:sc-2} and \ref{ln:sc-3} with $S_3\deq \smash{S_1^{-1}BS_1^{-*}}$.\;} \label{ln:sc-3a}
$S_3\deqr Q\Lambda Q^*$\quad (Schur decomposition)\label{ln:sc-4}\;
$S_4\deq Qf(\Lambda )Q^*$\;\label{ln:sc-5}
$S_5\deq S_1S_4S_1^* $\;\label{ln:sc-6}
\altline{\emph{Faster variant:} Replace Lines \ref{ln:sc-5}, \ref{ln:sc-6} with $S_t\deq S_1Q$, $S_5\deq S_tf(\Lambda)S_t^*$.\;}\label{ln:sc-7}
\Return{$S_5$}\;\label{ln:sc-8}
\end{algorithm2e}

This algorithm requires fewer flops because computing the Cholesky factorization of $A$ is cheaper than computing its square root.
Line \ref{ln:sc-1} requires $n^3/3$ flops, while the solution of the two linear systems on Line \ref{ln:sc-3a} requires $(1+1/3)n^3$ flops in total.
The two multiplications on Line \ref{ln:sc-6} similarly require $(1+1/3)n^3$ flops, while the dense, Hermitian matrix multiplication on Line \ref{ln:sc-5} costs $n^3$ flops. Assuming a cost of $9n^3$ flops for the Schur decomposition on Line~\ref{ln:sc-4}, \cref{alg:2} requires $13n^3$ flops, which is about half of the operations needed by \cref{alg:1}.
A slight reduction of $n^3/3$ flops can be achieved by
replacing Lines \ref{ln:sc-5} and \ref{ln:sc-6} by $S_t=S_1Q$ and $S_5=S_tf(\Lambda)S_t^*$,
respectively. This reduces the cost of Lines \ref{ln:sc-5} and \ref{ln:sc-6} from
$(2 + 1/3)n^3$ to $2n^3$ flops, resulting in an overall cost of
$(12 + 2/3)n^3$ flops.

If $B$ is also positive definite, then the computational cost of \cref{alg:2} can be further reduced. This version, reported in \cref{alg:3}, swaps a matrix multiplication for a
Cholesky decomposition. In this variant, Lines \ref{ln:scpsd-1} and \ref{ln:scpsd-2a} require $2n^3/3$ and $n^3/3$ flops, respectively. The Schur decomposition on Line \ref{ln:scpsd-4}
requires $9n^3$ flops, while computing each of $S_4$ and $S_5$ requires
$n^3$ flops. Therefore, the overall cost of the faster variant of \cref{alg:3} is $(12 + 1/3)n^3$ flops, slightly lower than the $13n^3$ flops of \cref{alg:2}.
As for \cref{alg:2}, the overall computational cost of \cref{alg:3} can be reduced by modifying Lines \ref{ln:scpsd-6} and \ref{ln:scpsd-7}; this yields a computational cost of $12n^3$ flops.
These minor improvements complicate the rounding error analysis, thus we prefer to keep Lines \ref{ln:sc-5} and \ref{ln:sc-6} in \cref {alg:2} and Lines \ref{ln:scpsd-6} and \ref{ln:scpsd-7} in \cref{alg:3}.

\begin{algorithm2e}[t]
\caption{Improved Cholesky--Schur algorithm for computing $Af(A^{-1}B).$}
\label{alg:3}
\KwIn{Positive definite matrices $A$ and $B$, and a real-valued function $f$.}
\KwOut{The {\color{revcol} Hermitian}  matrix $S_5 = Af(A^{-1}B)$.}
Compute the Cholesky decompositions $A\deqr R_A^*R_A$ and $B\deqr R_B^*R_B$ of $A$ and $B$, respectively.\; \label{ln:scpsd-1}
$S_1\deq R_A^{-*}$\; \label{ln:scpsd-2}
$S_2\deq S_1R_B^*$\; \label{ln:scpsd-3}
\altline{\emph{Faster variant:} Replace Lines \ref{ln:scpsd-2}, \ref{ln:scpsd-3} with $S_2\deq R_A^{-*}R_B^*$\;} \label{ln:scpsd-2a}
$S_3\deq S_2S_2^*$\; \label{ln:scpsd-4}
$S_3\deqr Q\Lambda Q^*$\quad (Schur decomposition)\; \label{ln:scpsd-5}
$S_4\deq Qf(\Lambda)Q^*$\; \label{ln:scpsd-6}
$S_5\deq R_A^*S_4R_A $\; \label{ln:scpsd-7}
\altline{\emph{Faster variant:} Replace Lines \ref{ln:scpsd-6}, \ref{ln:scpsd-7} with $S_t\deq R_A^*Q$, $S_5\deq S_tf(\Lambda)S_t^*$\;}
\Return{$S_5$}\; \label{ln:scpsd-8}
\end{algorithm2e}

In the remainder of this section, we compare the numerical stability of \cref{alg:1,alg:2} by analyzing their behavior in floating-point arithmetic. The main differences between the two algorithms lie in the method used to compute the quantities $S_1$ and $S_3$. The analysis of the error in $S_3$ will play a crucial role in the final error bound as well as the backward error analysis of $S_1$. The next subsection recalls some preliminary results needed in the error analysis.

Hereafter, for simplicity, we assume that the matrices have real entries, but the analysis remains valid in the case of complex arithmetic.
Finally, we require that $f$ be twice differentiable with a continuous second derivative, so that $f(x+\epsilon)-f(x)\doteq \epsilon f'(x)$ and $f'(x+\epsilon)-f'(x)\doteq \epsilon f''(x)$ as $\varepsilon\rightarrow 0$.

\subsection{Background results}

We recall some standard results concerning rounding error analysis of basic matrix operations. We refer the reader to the monographs by \citet{higham:book02, higham:book08} for more background and context on these results.

By definition, $\| |A|\|_F=\|A\|_F$, and one can show that $\|A\|_F\le \sqrt n\|A\|$, that $\|\hbox{vec}(A)\|=\|A\|_F$, and that $\|I\otimes A\|=\|A\otimes I\|=\|A\|$.
If $C=AB$, then taking the 2-norm of both sides of the identity $\opvec(C)=(I\otimes A)\opvec(B)=(B^T\otimes I)\opvec(A)$ gives
\begin{equation}\label{eq:F2}
\|C\|_F\le\|A\|\|B\|_F \quad\text{and}\quad \|C\|_F\le\|A\|_F\|B\|.
\end{equation}

Let $X=A^{1/2}$ be the square root
of a positive definite matrix $A$,
and let $\wt X=\hbox{fl}(A^{1/2})$ be the matrix that the Schur algorithm computes in floating-point arithmetic. The error analysis of \citet[p.~139]{higham:book08} yields
\begin{equation}\label{eq:ersqrt}
    \wt X^2=A+\Delta,\quad \|\Delta\|_F\ledot u\, c\, n^3 \|\wt X\|_F^2\doteq ucn^3\|X\|_F^2=ucn^3\|A\|_F,
\end{equation}
where $c$ is a small positive constant and $u$ is the unit roundoff of the working precision. In other words, the computed matrix $\wt X$ is the exact square root of a slightly perturbed matrix $A+\Delta$.

For the Cholesky factorization $A=R_A^*R_A$ of a positive definite matrix $A$, a direct inspection shows that $Z\deq{}R_A^{-*}A^{1/2}$ and
$Z^{-1}=A^{-1/2}R_A^*$ are both unitary. Since $R_A^{-*}=ZA^{-1/2}$, this fact implies that
 \begin{equation}\label{eq:nrm}
 \|R_A^{-*}\|=\|A^{-1/2}\|,\quad
\|R_A^*\|=\|A^{1/2}\|, \quad
\cond(R_A^*)=\cond(A^{1/2}).
\end{equation}
Moreover, if we denote by $\wt R_A$
the matrix obtained by computing the Cholesky factorization in floating-point arithmetic, we have \cite[Thm.~10.5]{higham:book02}
\begin{equation}\label{eq:chol0}
  \wt R_A^*\wt R_A= A+H,\quad |H|\ledot u(n+1)dd^*,\quad d_i=a_{ii}^{1/2}.
\end{equation}
Therefore, $\wt R_A$ is the exact Cholesky factor of a slightly perturbed matrix $A+H$.
Since \mbox{$\|A^{1/2}\|^2=\|A\|$}, $\|A^{1/2}\|_F^2\le n\|A\|$, and $\|dd^*\|_F=\|dd^*\|=d^*d=\hbox{trace}(A)=\|A^{1/2}\|_F^2$, by using \eqref{eq:chol0} we obtain
\begin{equation}\label{eq:cholD}
\|H\|_F=\|\wt R_A^*\wt R_A-A\|_F\ledot u(n+1)\|A^{1/2}\|_F^2\le u n (n+1)\|A\|.
\end{equation}

For the matrix multiplication $X=AB$, if we denote by $\wt X=\hbox{fl}(AB)$ the product computed in floating-point arithmetic, from \cite[eq.~(3.13)]{higham:book02} we have
\begin{equation}\label{eq:mm2}
|\wt X-X|\ledot u\,n |A|\cdot |B|,
\end{equation}
whence we get
$
\|\wt X-X\|_F\ledot u\,n\|A\|_F\|B\|_F\le u\,n^{3/2}\|A\|_*\|B\|_*$, where one of the two norms $\|\cdot\|_*$ is the Frobenius norm and the other is the $2$-norm.
A straightforward application of \eqref{eq:F2} yields
\begin{equation}\label{eq:mmE}
\|\hbox{fl}((AB)C)-ABC\|_F\ledot 2u\, n^{3/2}\|A\|_*\|B\|_*\|C\|_*,
\end{equation}
where one of the three $\|\cdot\|_*$ above is $\|\cdot\|_F$ and the remaining two are the $2$-norm.

We can compute $T^{-1}B$, for $T$ triangular, by substitution. Let $\hbox{fl}(T^{-1}B)$ denote the value obtained by running the back-substitution algorithm in floating-point arithmetic. From \cite[Thm.~8.10]{higham:book02}, we have
\begin{equation}\label{eq:tinv}
\hbox{fl}(T^{-1}B)=T^{-1}B+E_1,\quad |E_1|\ledot u\,\sigma_n M(T)^{-1}|B|,\quad \sigma_n=n^2+n+1,
\end{equation}
where $M(T)$ is the M-matrix whose entries have the same moduli as the corresponding entries of $T$.
From \eqref{eq:F2}, we deduce that
$\|E_1\|_F\ledot u\,\sigma_n\|M(T)^{-1}\|\cdot \|B\|_F\le u\sigma_n\sqrt n\|M(T)^{-1}\| \cdot\|B\|$.
For $CT^{-*}$ we analogously obtain $\hbox{fl}(CT^{-*})=\hbox{fl}(T^{-1}C^*)^*=(T^{-1}C^*+E_2)^*$, where
$|E_2|\ledot u\,\sigma_nM(T)^{-1}|C^*|$.

Setting $C=T^{-1}B$, we can bound the error in the computation of $(T^{-1}B)T^{-*}=CT^{-*}$. We have
\[\begin{split}
&\hbox{fl}\bigl((T^{-1}B)T^{-*}\bigr)\doteq (T^{-1}B+E_1)T^{-*}+E_2^*
=T^{-1}BT^{-*}+E_1T^{-*}+E_2^*,\\
&|E_1|\ledot u\,\sigma_nM(T)^{-1}|B|,  \qquad
|E_2|\ledot u\,\sigma_n M(T)^{-1}|T^{-1}B|^*.
\end{split}
\]
In other words,
$
\hbox{fl}\bigl((T^{-1}B)T^{-*}\bigr)=T^{-1}BT^{-*}+E$, where
$E=E_1T^{-*}+E_2^*$.
Using \eqref{eq:F2}, we get
\begin{equation}\label{eq:tbt}
    \|\hbox{fl}\bigl((T^{-1}B)T^{-*}\bigr)-T^{-1}BT^{-*}\|_F \ledot
    2u\,\sigma_n\|T^{-1}\|\cdot \|M(T)^{-1}\|\cdot \|B\|_F.
\end{equation}

For $B=I$, equation \eqref{eq:tinv} becomes
\begin{equation}\label{eq:Tinv0}
|\hbox{fl}(T^{-1})-T^{-1}|\ledot u\,\sigma_n M(T)^{-1},
\end{equation}
and taking norms yields
\begin{equation}\label{eq:Tinv}
\|\hbox{fl}(T^{-1})-T^{-1}\|_F\ledot u\, \sigma_n\|M(T)^{-1}\|_F\le u\, \sqrt n\sigma_n\|M(T)^{-1}\|.
\end{equation}

Consider the inversion of a positive definite matrix $A$ using the identity $A^{-1}=R_A^{-1}R_A^{-*}$, where $R_A$ is the Cholesky factor of $A$. Let $\wt R_A$ denote the computed Cholesky factor that satisfies $\wt R_A^*\wt R_A=A+H$ in \eqref{eq:chol0} and \eqref{eq:cholD}. If $\wt V \deq{} \hbox{fl}(\wt R_A^{-1})$, then by \eqref{eq:Tinv0} we have $\wt V=\wt R_A^{-1} +E$, where
$|E|\ledot u\, \sigma_n M(R_A)^{-1}$. Moreover, denote $\fl(A^{-1})=\fl(\wt V \wt V^*)$.
Relying on \eqref{eq:mm2}, we find that
\[
\hbox{fl}(A^{-1})=\hbox{fl}( \wt V \wt V^*)=(\wt R_A^{-1}+E)(\wt R_A^{-*}+E^*)+W,\quad |W|\ledot u\, n|\wt V|\cdot|\wt V^*|\doteq u\, n |\wt R_A^{-1}|\cdot|\wt R_A^{-*}|,
\]
so that $G\deq{}\fl(A^{-1})-A^{-1}\doteq R_A^{-1}E^*+ER_A^{-*}-A^{-1}HA^{-1}+W$.
Taking norms, we get
\[
    \|G\|_F \ledot 2\|R_A^{-1}\|\cdot\|E\|_F+\|A^{-1}\|^2\cdot\|H\|_F+\|W\|_F.
\]
Since $\|E\|_F\ledot u\,\sigma_n\|M(R_A)^{-1}\|_F$, $\|W\|_F\ledot u\,n\|\wt R_A^{-1}\|_F^2$, and $\|R_A^{-1}\|^2=\|A^{-1}\|$,
from \eqref{eq:cholD} we obtain
\begin{equation}\label{eq:boundM}
\begin{split}
    &\|G\|_F\ledot u\, \|A^{-1}\|\bigl(n(n+1)\cond (A)+2\sigma_n\psi(R_A)+n^2\zeta\bigr),\\
    &\psi(R_A) \deq\frac{\|M(R_A)^{-1}\|_F}{\|R_A^{-1}\|}, \quad
    \zeta=\frac{\|\wt R_A^{-1}\|^2}{\|A^{-1}\|}.
\end{split}
\end{equation}
Observe that, since $\wt R_A$ is an approximation of $R_A$ and $\|R_A^{-1}\|^2=\|A^{-1}\|$, one should expect that $\zeta\approx 1$.

\subsection{Rounding error analysis of \cref{alg:1}}\label{sec:alg1}
{
The $i$th step of \cref{alg:1} has the form $S_i = s_i(S_1,\ldots,S_{i-1},A,B)$, where $s_i(\cdot)$ is the computation performed at step $i$.  We denote by $\widetilde S_i$ the matrix obtained after performing step $i$ in floating-point arithmetic and by $E_i = \widetilde S_i - s_i(\widetilde S_1,\ldots,\widetilde S_{i-1},A,B)$ the corresponding error at step $i$. We use an analogous notation for the decompositions on Line \ref{ln:ss-4} of \cref{alg:1} and on Lines \ref{ln:sc-1} and \ref{ln:sc-4} of \cref{alg:2}.

In floating-point arithmetic, the steps of \cref{alg:1} can be written as follows.
\begin{enumerate}
    \item $\widetilde S_1=A^{1/2}+E_1$
    \item $\widetilde S_2=\widetilde S_1^{-1}+E_2$\label{it:mstep-2}
    \item $\widetilde S_3=\widetilde S_2B\widetilde S_2^*+E_3$
    \item If $\widetilde S_3\deqr{}\widehat Q \widehat \Lambda  \widehat Q^*$ is the exact Schur decomposition of $\widetilde S_3$, then we denote the factors actually computed in floating-point arithmetic by $\wt Q$, $\wt \Lambda$, respectively.\label{it:mstep-4}
    \item $\widetilde S_4=\widetilde Q(f(\widetilde \Lambda )+E_f)\widetilde Q^*+E_4$
    \item $\widetilde S_5=\widetilde S_1\widetilde S_4 \widetilde S_1^*+E_5$\label{it:mstep-6}
\end{enumerate}
For $i = 1, \dots, 5$, the matrix $E_i$ accumulates the rounding errors occurring at step $i$. The Hermitian matrices $E_1$, $E_2$, and $E_3$ contain the errors occurring when computing the square root of $A$, the inverse of $\widetilde S_1$, and the product $\widetilde S_2B\widetilde S_2$, respectively.
{\hilight We define the matrix $E_\Lambda $ as
\begin{equation}\label{eq:eq}
E_\Lambda \deq{} \wt \Lambda -\wh \Lambda,
\end{equation}}
where $\widetilde S_3=\widehat Q\widehat \Lambda \widehat Q^*$ is the exact Schur decomposition of $\widetilde S_3$, and $\wt Q$ and $\wt \Lambda$ are the factors computed in floating-point arithmetic.
We denote by $\lambda_i,\wh\lambda_i$, and $\wt \lambda_i$ the diagonal entries of the diagonal matrices $\Lambda$, $\wh\Lambda$, and $\wt\Lambda$, respectively.
Therefore, $E_\Lambda$ is also diagonal.
Finally, $E_f$ is diagonal, and its entries are the errors in the floating-point computation of $f({\hilight \wt \lambda_i})$, while $E_4$ and $E_5$ collect the errors in the floating-point computation of $\widetilde S_4$ and $\widetilde S_5$, respectively.

In practice, at step $i$ we want to bound the difference between $\widetilde S_i$, obtained by running the algorithm in floating-point arithmetic, and the exact, unknown value $S_i$; we denote this quantity by $F_i\deq{}\widetilde S_i-S_i$.

{\hilight For the Schur decomposition in Step \ref{it:mstep-4}, we rely on a general result, which we now recall in a modified form that better suits our notation.

\begin{theorem}{{\cite[Theorem 3.3]{drma20}}}\label{th:eigerr}
Let $X$ be a Hermitian matrix, and let $X=QDQ^*$, with $D$ diagonal and $Q$ unitary, be its spectral decomposition. Denote by $\wt Q$ and $\wt D$ the matrices obtained in floating-point arithmetic by applying one of the classical methods: either a tridiagonalization-based technique, which reduces the matrix to tridiagonal form and then applies a tridiagonal eigenvalue solver, or a Jacobi-type algorithm. Then, there exists a unitary matrix $\wc Q$ such that $X+\wc \Delta=\wc Q\wt D\wc Q^*$, where $\|\wc\Delta\|\le u\,p(n)\|X\|$ and $\|\wt Q-\wc Q\|\le u\,\wc\gamma\, n$ for a positive constant $\wc\gamma$ and a mildly growing function $p(n)$ that depends on the method used.
\end{theorem}

\begin{remark}\label[remark]{rem:drma}
Since $\|\check Q\|=1$, and $\|\wt Q-\wc Q\|\le u\,\wc\gamma\, n$, applying norms to both sides of the identities $\wt Q=\wc Q+(\wt Q-\wc Q)$ and $\wc Q=\wt Q-(\wt Q-\wc Q)$ yields $1-u\wc\gamma n\le\|\wt Q\|\le 1+u\wc\gamma n$, that is, $\|\wt Q\|\doteq 1$. From Weyl's theorem, we can conclude that $\|\wt D-D\|\le up(n)\|X\|$.
\end{remark}}
Finally, we assume that evaluating $f(x)$ at $x=\lambda$ is backward stable, so that $|\hbox{fl}(f(\lambda))-f(\lambda)|\le u\, \xi |f'(\lambda)| $ for some positive constant $\xi$.

In view of the assumptions above and of bounds
\eqref{eq:mmE} and \eqref{eq:boundM}, we have the error bounds
\begin{equation}\label{eq:as1}
   \begin{array}{ll}
    \|E_2\|_F\ledot u\,\|A^{-1/2}\|(\theta_1\cond(A)^{1/2}+2\sigma_n\psi(R) +\zeta n^2),\qquad
    &\|E_4\|_F\ledot u\,\theta_2\|f(\wh \Lambda) \|_F,\\[1ex]
    \|E_3\|_F\ledot u\,\theta_2\|S_2\|^2\!\!\cdot \|B\|_F=u\,\theta_2\|A^{-1}\|\cdot \|B\|_F,
    &\|E_5\|_F\ledot u\,\theta_2\|A\|\cdot \|f(\wh \Lambda)\|_F,
 \end{array}
\end{equation}
where $R$ is the Cholesky factor of $A^{1/2}$, $\psi(R) \deq \|M(R)^{-1}\|_F/\|R^{-1}\|$, $\theta_1=n(n+1)$,
and  $\theta_2=2n^{3/2}$. To bound $\|E_5\|_F$, we have used {\hilight \eqref{eq:mmE}
and
\begin{equation}\label{eq:as1bis}
 \begin{array}{lll}
    \|E_f\|_F\ledot u\, \xi \|f'(\wh \Lambda )\|_F,
 \end{array}
\end{equation}
which follows from the backward stability of the computation of $f$.}

The diagonals of $\Lambda$ and $\wh\Lambda$ contain the eigenvalues of $S_3$ and $\wt S_3=S_3+F_3$, respectively. Equation \eqref{eq:alg1-errors} below shows that $\|F_3\|=O(u)$, and by Weyl's inequality, we can conclude that

\begin{equation}\label{eq:lambda}
\|\wh\Lambda-\Lambda\|\le \|F_3\|=O(u)
\end{equation}
so that we may replace $\wh\Lambda$ in \eqref{eq:as1} and  \eqref{eq:as1bis} with $\Lambda$.

We are ready to perform the rounding error analysis of \cref{alg:1}. First of all, observe that since $\wt S_1^{\;2}=A+\Delta$, we may interpret the floating-point output of \cref{alg:1} as the result obtained by running Steps \ref{it:mstep-2}--\ref{it:mstep-6} in floating-point arithmetic to evaluate $\varphi(\wt A,B)=\wt Af(\wt A^{-1}B)$, where $\wt A=A+\Delta$ and $\|\Delta\|_F$ is bounded in \eqref{eq:ersqrt}. In other words, we may assume that the error $E_1$ in the floating-point computation  of $A^{1/2}$ is zero. Thus, the overall error is given by the sum of two components: the perturbation error
\[
    \mathcal D_1=\|\varphi(\wt A,B)-\varphi(A,B)\|_F,\quad \|\wt A-A\|_F=\|\Delta\|_F\leq u\, c\, n^3\|A^{1/2}\|_F^2,
\]
and the error $F_5=\wt S_5-S_5$ obtained with $F_1=0$.
We now analyze these two components separately.

The relative perturbation error $\mathcal D_1/\|\varphi(A,B)\|_F$ can be bounded by the product of the relative perturbation $\|\Delta\|_F/\|A\|_F$ and the condition number $\hbox{cond}(\varphi;A,B)$ of the matrix function $\varphi(A,B)$ analyzed in \cref{sec:conditioning}. Therefore,
\begin{equation}\label{eq:D}
\mathcal D_1\ledot\mathcal E_0,\quad \mathcal E_0= \|S_5\|_F\|\Delta\|_F\hbox{cond}(\varphi;A,B)/\|A\|_F.
\end{equation}

To analyze the error $F_5=\wt S_5-S_5$, a simple formal manipulation using the Neumann series, the definition $F_i=\widetilde S_i-S_i$, and the assumption that $F_1=E_1=0$,  gives
\begin{equation}\label{eq:alg1-errors}
\begin{aligned}
    F_2 &= E_2,\\
    F_3 &= E_3+(S_2+F_2)B(S_2^*+F_2^*)-S_3\doteq E_3+F_2BS_2^*+S_2BF_2^*.
\end{aligned}
\end{equation}
In view of \eqref{eq:F2}, taking the Frobenius norm of~\eqref{eq:alg1-errors} yields
$\|F_3\|_F\ledot \|E_3\|_F+2\|B\|\cdot\|A^{-1/2}\|\cdot\|E_2\|_F$,
and by using \eqref{eq:as1} we can write
\begin{equation}\label{eq:F3}
\begin{split}
&\|F_3\|_F\ledot  u\,\theta_2\|A^{-1}\|\cdot\|B\|_F+2\|B\|\cdot\|A^{-1/2}\|\cdot\|E_2\|_F
{\color{revcol} \ledot u\|A^{-1}\|\cdot \|B\| \chi}\\
&{\color{revcol}\chi=(2\theta_1\cond(A)^{1/2}+4\sigma_n\psi(R)+
{\color{revcol} (2\zeta +2)n^2}
)},
\end{split}
\end{equation}
where $R$ is the Cholesky factor of $A^{1/2}$,
$\psi$ {\color{revcol} and $\zeta$ are} defined in \eqref{eq:boundM},
and we have used the fact that any positive definite $X \in \Cnn$ satisfies $\| X^{1/2} \|^2 = \|X\|$.

A more detailed analysis is needed for $F_4=\wt S_4-S_4$.

{\hilight
Let $\wc Q$ and $\wc\Delta$ be the matrices obtained by applying \cref{th:eigerr}
with $X=\wt S_3$. Then $\|\wc\Delta\|\le u\, p(n)\|\wt S_3\|$ and the difference matrix
$W:=\wt Q-\wc Q$ satisfies $\|W\|\le u\, \wc\gamma\, n$.
Since $F_4=\wt S_4-S_4=E_4+\wt Q(f(\wt\Lambda)+E_f)\wt Q^*-S_4$,
$\wt Q=\wc Q+W$, and $S_4=f(S_3)$,
we can conclude that
\begin{equation}\label{{eq:alg1-errors1}}
\begin{split}
F_4&=E_4+(\wc Q+W)(f(\wt\Lambda)+E_f)(\wc Q^*+W^*)-S_4\\
   &\doteq E_4 +
   Wf(\wt\Lambda)\wc Q^*+\wc Qf(\wt\Lambda)W^* +\wc Q E_f\wc Q^*+\wc Qf(\wt\Lambda)\wc Q^*-f(S_3) \\
   &\doteq E_4+
   Wf(\wt\Lambda)\wc Q^*+\wc Qf(\wt\Lambda)W^* +\wc Q E_f\wc Q^* +f(S_3+F_3+\wc\Delta)-f(S_3)
,
\end{split}
\end{equation}
In the last step, we used the identities
$\wc Qf(\wt\Lambda)\wc Q^*=f(\wc Q\wt\Lambda\wc Q^*)$ and
$\wc Q\wt\Lambda\wc Q^* = \wt S_3+\wc\Delta =S_3+F_3+\wc\Delta$,
which follows from \cref{th:eigerr}.

Combining \eqref{eq:lambda} and \cref{rem:drma}, we obtain that $\|\wt\Lambda-\Lambda\|\le \|\wt\Lambda-\wh\Lambda\|+\|\wh\Lambda-\Lambda\|=O(u)$.
Then, $Wf(\wt\Lambda)\doteq Wf(\Lambda)$ and, similarly, $f(\wt\Lambda)W^*\doteq f(\Lambda)W^*$, so that we may rewrite \eqref{{eq:alg1-errors1}} as
\begin{equation}\label{eq:alg1-f}
F_4 \doteq E_4+
   Wf(\Lambda)\wc Q^*+\wc Qf(\Lambda)W^* +\wc Q E_f\wc Q^* +f(S_3+F_3+\wc\Delta)-f(S_3).
\end{equation}

In \eqref{eq:alg1-f} we can use $f(S_3+F_3+\wc\Delta)-f(S_3)\doteq Df(S_3)[F_3+\wc\Delta]$,
where $Df(S_3)[F_3+\wc\Delta]$ is the Fréchet derivative of $f(X)$ at $X=S_3$ in the direction $F_3+\wc\Delta$, to arrive at the following expression
\begin{equation}\label{eq:alg1-errors2}
F_4\doteq E_4+Wf(\Lambda)\wc Q^*+\wc Qf(\Lambda)W^* +\wc Q E_f\wc Q^* +Df(S_3)[F_3+\wc\Delta].
\end{equation}
}

Finally,
for $F_5=\widetilde S_5-S_5$ we get
\begin{equation}\label{eq:alg1-errors3}
\begin{split}
F_5&=E_5+S_1(S_4+F_4)S_1^*-S_5
= E_5+S_1F_4S_1^*.
\end{split}
\end{equation}

Taking the Frobenius norm on both sides, and using \eqref{eq:F2} and $\|S_1\|\cdot\|S_1^*\|=\|A\|$, we get the inequality $\|F_5\|_F\ledot \|E_5\|_F+\|A\|\|F_4\|_F$. Using \eqref{eq:alg1-errors2} together with $\|\wc Q\|=\|\wc Q^*\|=1$, we obtain
\begin{equation}\label{eq:f5-final}
\begin{split}
\|F_5\|_F&\ledot \|E_5\|_F+\|A\|\bigl(\|E_4\|_F+2{\hilight\|W\| }\cdot\|f(\Lambda)\|_F
+\|E_f\|_F\bigr)
+\|A\|\cdot \|\Df{f}{S_3}{F_3{\hilight+\wc\Delta } }\|_F.
\end{split}
\end{equation}
{\hilight Now, we provide two constants, $\mathcal E_1$ and $\mathcal E_2$, that bound the sum of the first two terms and the third term in \eqref{eq:f5-final}, respectively.} Using \eqref{eq:as1} in the above expression together with \eqref{eq:as1bis}, \eqref{eq:lambda}, \eqref{eq:F3}, {\hilight and the bound  $\|W\|\le u\,\wc\gamma\, n$}, we get
\begin{equation}\label{eq:bound1}
\begin{split}
    &\|F_5\|_F\ledot \mathcal E_1+\mathcal E_2,\\
     &\mathcal E_1={\hilight
     2u(\theta_2+\wc\gamma\, n)}\|A\|\cdot\|f(\Lambda)\|_F +u\,{\hilight \xi\|A\|}\cdot\|f'(\Lambda)\|_F,
     \\
    &\mathcal E_2=u\,\mu(A)
    \|F\|_F \cdot \|B\| (\chi {\hilight+\sqrt n p(n)),}
\end{split}
\end{equation}
where  $F=(f[\lambda_i,\lambda_j])_{ij}$, $\lambda_i$ are the eigenvalues of $S_3=A^{-1/2}BA^{-1/2}$ and the quantity {\color{revcol} $\chi$ is defined in \eqref{eq:F3}}.
To see that $\mathcal E_2$ bounds $\|A\|\cdot \|\Df{f}{S_3}{{\hilight Z}}\|_F$, {\hilight for $Z=F_3+\wc\Delta$}, note that \eqref{eq:8} and the Daleckiĭ--Kreĭn theorem {\hilight \eqref{eq:9}} imply that $\Df{f}{S_3}{{\hilight Z}}=Q(F\circ (Q^*{\hilight Z}Q))Q^*$, where $Q^*Q=I$, and thus
\begin{equation}\label{eq:nfd}
\begin{split}
\|\Df{f}{S_3}{{\hilight Z}}\|_F
 = \|Q(F\circ (Q^*{\hilight Z}Q))Q^*\|_F =\|F\circ (Q^*{\hilight Z}Q)\|_F \le  \|F\|_F\cdot \|{\hilight Z}\|_F.
\end{split}
\end{equation}
Here, we have used the property that $\|F\circ (Q^*{\hilight Z}Q)\|\le\|F\|\cdot\|Q^*{\hilight Z}Q\|=\|F\|\cdot\|{\hilight Z}\|$,
which holds for any unitarily invariant norm~\cite[Fact~11.10.96]{bern18}.
{\color{revcol} Finally, using in \eqref{eq:nfd} {\hilight the bounds \eqref{eq:F3} and
\mbox{$\|\wc\Delta\|_F\le \sqrt n\|\wc\Delta\|\ledot u\, \sqrt n p(n)\|S_3\|\le u\, \sqrt n p(n)\|A^{-1}\|\cdot\|B\|$}} provides the expression for $\mathcal E_2$}.

Therefore, for the overall error, we get the bound
\begin{equation}\label{eq:bound1bis}
\mathcal D_1 + \|F_5\|_F\le \mathcal E_0+\mathcal E_1+\mathcal E_2,
\end{equation}
where $\mathcal E_0$ is defined in \eqref{eq:D} and $\mathcal E_1$ and $\mathcal E_2$ are defined in \eqref{eq:bound1}.

\subsection{Rounding error analysis of \cref{alg:2,alg:3}}\label{sec:alg2}
The error analysis of \cref{alg:2} is similar.
The two algorithms differ in the computation of the matrix~$S_1$, which is the square root of $A$ in \cref{alg:1} and the Cholesky factor of $A$ in \cref{alg:2}.

Let $A\deqr{}R_A^*R_A$ be the Cholesky decomposition of $A$ and let $\wh R_A$ be the matrix computed in floating-point arithmetic by the Cholesky algorithm. We may use \eqref{eq:chol0} to interpret the floating-point result of \cref{alg:2} as the result obtained by evaluating $\varphi(\wt A,B)=\wt Af(\wt A^{-1}B)$ at \mbox{$\wt A \deq{} \wh R_A^*\wh R_A=A+H$} in floating-point arithmetic.
The final error is then the sum of two terms: the perturbation error
\begin{equation}\label{eq:Hbis}
\mathcal D_2=\|\varphi(\wt A,B)-\varphi(A,B)\|_F,\quad \hbox{where} \quad \|H\|_F=\|\wt A-A\|_F\ledot u\, (n+1)\| A \|_F,
\end{equation}
and the roundoff term $F_5 = \wt S_5 - S_5$, obtained with $F_1=0$, {\hilight $S_1=R_A^*$, and $S_2=R_A^{-*}$}.
Therefore, we may repeat the analysis performed for \cref{alg:1}.
If we replace $\Delta$ by $H$ in \eqref{eq:D}, we obtain
\begin{equation}\label{eq:Dbis}
\mathcal D_2\ledot \mathcal E_0',\quad \mathcal E_0'=\|S_5\|_F\|H\|_F
\operatorname{cond}(\varphi;A,B)/
\|A\|_F\le
u(n+1)\|S_5\|_F\cond(\varphi;A,B)
.
\end{equation}

The bound on $F_5=\wt S_5-S_5$ is also analogous.
We compute $S_3$ using the faster method described on Line \ref{ln:sc-3a} of \cref{alg:2}, without forming
$S_2=S_1^{-1}$ but applying back substitution in
$S_3=S_1^{-1}BS_1^{-*}$.
As $S_2$ is no longer computed, we have $E_2=F_2=0$.
Then, from \eqref{eq:tbt} with the bound rewritten in the form $2u\,\sigma_n\|T^{-1}\|\cdot\|M(T)^{-1}\|_F\cdot\|B\|$, we obtain
\begin{equation}\label{eq:F3_}
\|F_3\|_F=\|E_3\|_F\ledot 2 u\,\sigma_n\|R_A^{-1}\|\cdot\|M(R_A)^{-1}\|_F\cdot \|B\|=
2u\,\sigma_n \|A^{-1}\|\cdot \|B\|\psi(R_A),
\end{equation}
where $\psi$ is defined in \eqref{eq:boundM}.
In~\eqref{eq:F3_}, we have used the identity \eqref{eq:nrm} and the fact that $\|R_A\|=\|A\|^{1/2}$.
Equation \eqref{eq:alg1-errors2} remains valid for \cref{alg:2}, and bounding $\|F_3\|_F$ with \eqref{eq:F3_}~yields
\begin{equation}\label{eq:bound2}
\begin{split}
\|F_5\|_F\ledot
\mathcal E_1+\mathcal E_2', \qquad
\mathcal E_2'= u\,\cond(A)  \|B\|
\cdot\|F\|_F( 2\sigma_n\psi(R_A)  {\hilight +\sqrt n p(n)})
\end{split}
\end{equation}
Thus, the total error for \cref{alg:2} is bounded by
\begin{equation}\label{eq:bound2bis}
\|F_5\|_F+\mathcal D_2\ledot\mathcal E_0'+\mathcal E_1+\mathcal E_2',
\end{equation}
where $\mathcal E_0'$ is defined and bounded in \eqref{eq:Dbis}, $\mathcal E_1$ is defined in \eqref{eq:bound1}, and {\hilight $\mathcal E_2'$} in \eqref{eq:bound2}.

Next, we compare the upper bounds \eqref{eq:bound1bis} and \eqref{eq:bound2bis} on the error generated by \cref{alg:1,alg:2}, respectively, which differ in two of the three quantities on the right-hand side. {\hilight From \eqref{eq:bound2} and \eqref{eq:bound1}, we} have
\begin{equation*}
    r_0 \deq{} \frac{\mathcal E_0}{\mathcal E_0'} = \frac{\|\Delta\|_F}{\|H\|_F},
    \qquad
    r_1 \deq{} \frac{\mathcal E_2}{\mathcal E_2'}
    {\hilight =\frac{\chi+t}{2\sigma_n\psi(R_A)+t},\quad t=\sqrt n p(n).}
\end{equation*}
{\hilight Indeed, $r_1\ge 1$ if $\chi\ge 2\sigma_n\psi(R_A)$.
  Since $\sigma_n=\theta_1+1$, we get
\[
\frac {\chi}{2\sigma_n\psi(R_A)}>\frac{\theta_1\mu(A)^{1/2}}{(\theta_1+1)\psi(R_A)}.
\]
That is, $r_1$ grows as $\mu(A)^{1/2}$ when $\psi(R_A)$ is close to 1.

For $r_0$,} from the bound \eqref{eq:cholD},  we get $r_0\ge \|\Delta\|_F/(u(n+1)\|A\|_F)$. We do not have a lower bound on $\|\Delta\|_F$, but assuming that the bound \eqref{eq:ersqrt} on $\|\Delta\|_F$ is sharp, we get the heuristic estimate $r_0 \approx cn^3/(n+1)$. In other words, $r_0$ grows as $n^2$ and is expected to get worse for large values of $n$.

This suggests that, numerically, \cref{alg:2} will behave better than~\cref{alg:1} {\color{revcol} for large values of $n$ and $\cond(A)$.}

The function $\psi$, which appears in both error bounds, is at least 1 and can be arbitrarily large~\cite[Prob.~8.2]{higham:book02}, but this is not always the case: if $A \in \Cnn$ is an M-matrix, for example, then $\psi(R_A) \le\sqrt n$.

The analysis for \cref{alg:3} is similar. The rounding errors in the Cholesky decomposition of $B$ can be viewed as the result of a perturbation on $B$, so that the perturbation error in \eqref{eq:Hbis} becomes
\[
\mathcal D=\|\varphi(\wt A,\wt B)-\varphi(A,B)\|,\quad \|\wt A-A\|\ledot nu\||R_A^*|\|^2,\quad \|\wt B-B\|\ledot nu\||R_B^*|\|^2,
\]
where $R_A$ and $R_B$ are the Cholesky factors of $A$ and $B$, respectively.
We omit the analysis which is entirely analogous to that of the two cases already discussed.

\section{Numerical experiments}
\label{sec:num-exp}

We now compare the algorithms presented in this work in terms of accuracy and performance.
Our experiments were run using the GNU/Linux version of MATLAB 2021b on a machine equipped with 16 GiB of RAM and a 4-core Intel i3-7100 processor running at 3.9 GHz.

We consider five implementations.
\begin{itemize}
    \item \algnaive: the algorithm obtained with the MATLAB command \verb|S = A * f (A \ B)|;
    \item \algss: an implementation of~\cref{alg:1};
    {\color{revcol}
    \item \algscsc: an implementation of~\cref{alg:ref};
    }
    \item \algsc: an implementation of~\cref{alg:2};
    \item \algscp: an implementation of~\cref{alg:3};
\end{itemize}
\Cref{alg:1,alg:2,alg:3} have been implemented in their faster variants.
To gauge accuracy, we use the relative forward error
\begin{equation}
    \label{eq:rel-fwd-err}
    \frac{\| \wt S - S \|_F}{\| S \|_F},
\end{equation}
where $\wt S$ is the solution computed by one of our implementations and $S$ is a reference solution computed in high precision arithmetic with the Advanpix Multiprecision Computing Toolbox for MATLAB using the default precision of 34 decimal digits, which corresponds to binary128 precision~\citep{ieee19}.

For a given size \texttt{n} and condition number \texttt{cnd}, the matrices $A$ and $B$ are generated randomly with the following code:
{\color{revcol}}
\begin{lstlisting}[language=matlab]
    a = (1 / cnd - 1) / (n - 1);
    d = [1:n] * a + (1 - a);
    [Q, ~] = qr(rand(n) - rand(n));
    A = Q * diag(d) * Q';
\end{lstlisting}
This ensures that the $n$ eigenvalues of $A$ are uniformly distributed in the closed interval $[\texttt{cnd}^{-1}, 1]$.
}

To obtain statistically significant results, for each configuration of the parameters
{\ttfamily n} and {\ttfamily cnd} we repeated the computation with 50 pairs of random matrices,
reporting the median of timings and errors. {\hilight The random number generator was initialized with the command {\ttfamily rng(0)}.}
To prevent large growth in the entries of the result, in all our experiments, we used $f(x) = \log(x)$.
{\hilight All the software used for the experiments, including our implementations of Algorithms 1--4, the function that generates the random matrices, and the functions that produce the data for the figures, is available online.\footnote{\url{https://people.dm.unipi.it/bini/Software/afa1b/}}}

\begin{figure}
    \centering
    \begin{tikzpicture}[trim axis left, trim axis right]
    \begin{axis}[
    xlabel={$n$},
    legend pos=south east, 
    ymode = log,
    log base y=10,
    ymin = 5e-17, ymax = 1e-2,
    xmin = 0, xmax = 210,
    max space between ticks=30,
    cycle list name = methodlist
    ]
    \addplot table[x=n, y=naive] 
    {accuracy_sizeA.dat};
    \addlegendentry{\algnaive}
    \addplot table[x=n, y=sqrt_schur] 
    {accuracy_sizeA.dat}; \addlegendentry{\algss}
    \addplot table[x=n, y=cholesky_schur] 
    {accuracy_sizeA.dat}; \addlegendentry{\algsc}
    \addplot table[x=n, y=cholesky_schur_psd] 
    {accuracy_sizeA.dat}; \addlegendentry{\algscp}
    \addplot table[x=n, y=schur_schur] 
    {accuracy_sizeA.dat}; \addlegendentry{\algscsc}     
    \end{axis}
    \alttext{Log-scale relative forward error against matrix size, for five algorithms with f(x) = log⁡(x). The matrices have size between 10 and 200. The matrix A is ill conditioned, with condition number 10 to the power 8, while B has condition number 10. The two Cholesky--Schur algorithms are the most accurate and have nearly identical errors. The square root--Schur algorithm is less accurate, while the naive and double-Schur algorithms generally have the largest errors.}
\end{tikzpicture}
    \hskip 40pt
   \begin{tikzpicture}[trim axis left, trim axis right]
    \begin{axis}[
    xlabel={$n$},
    legend pos=south east, 
    ymode = log,
    log base y=10,
    ymin = 5e-17, ymax = 1e-2,
    xmin = 0, xmax = 210,
    max space between ticks=30,
    cycle list name = methodlist
    ]
    \addplot table[x=n, y=naive] 
    {accuracy_sizeAB.dat};
    \addlegendentry{\algnaive}
    \addplot table[x=n, y=sqrt_schur] 
    {accuracy_sizeAB.dat}; \addlegendentry{\algss}
    \addplot table[x=n, y=cholesky_schur] 
    {accuracy_sizeAB.dat}; \addlegendentry{\algsc}
    \addplot table[x=n, y=cholesky_schur_psd] 
    {accuracy_sizeAB.dat}; \addlegendentry{\algscp}
    \addplot table[x=n, y=schur_schur] 
    {accuracy_sizeAB.dat}; \addlegendentry{\algscsc}     
    \legend{}
    \end{axis}
    \alttext{Log-scale relative forward error against matrix size n, for five algorithms with f(x) = log⁡(x). The matrices have size between 10 and 200 and are both ill conditioned, with condition number 10 to the power 8. The two Cholesky--Schur algorithms have nearly identical errors and are several order of magnitudes more accurate of the remaining algorithms. The square root--Schur algorithm is less accurate, the naive algorithm is slightly less accurate for large n, and double-Schur algorithm has the largest errors.}
\end{tikzpicture}
    \caption{Relative forward error~\eqref{eq:rel-fwd-err} of the different implementations on random positive definite $n\times n$ matrices of increasing size $n$, with $f(x)=\log(x)$. On the left, $A$ is ill conditioned ($\cond(A)=10^8$), while $B$ is not ($\cond(B)=10$); on the right, both matrices are ill conditioned ($\cond(A)=\cond(B)=10^8$). The reported values are the median of 50 measurements across different matrices.}
    \label{fig:experiment1}
\end{figure}
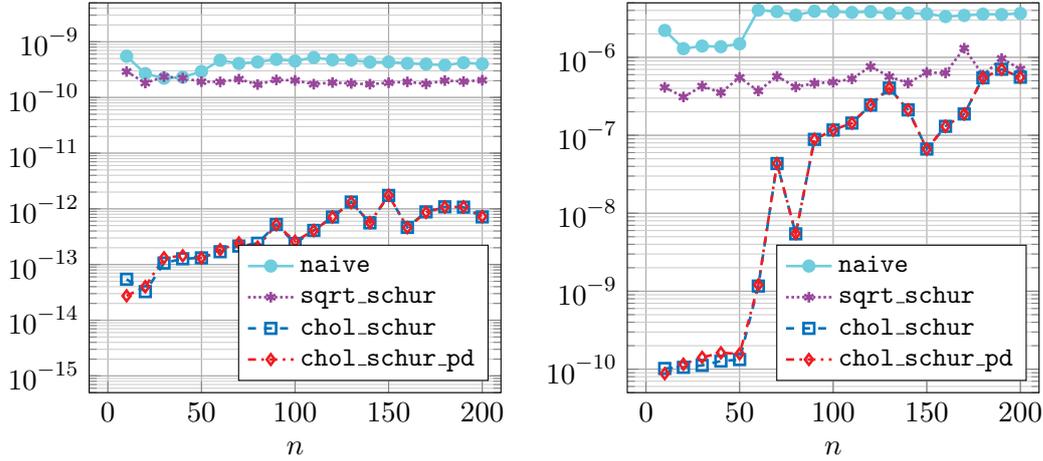

\Cref{fig:experiment1} reports the relative error~\eqref{eq:rel-fwd-err} as $n$ increases from 10 to 200 with step 10. In the panel on the left, $A$ and $B$ have condition numbers of $10^8$ and $10$, respectively, making $ A$ ill conditioned and $ B$ well conditioned.  In the panel on the right, both matrices are ill conditioned, with $\cond(A)=\cond(B)=10^8$.
In both cases, \algnaive{} and \algscsc{} are the algorithms most sensitive to the conditioning of the problem, while \algss{} has a comparable behavior but is typically slightly more accurate.
The errors of \algsc{} and \algscp{} are almost indistinguishable, and these two algorithms achieve a much better accuracy, especially when $A$ and $B$ are not well conditioned.

\begin{figure}
    \centering
    \begin{tikzpicture}[trim axis left, trim axis right]
    \begin{axis}[
    xlabel={$\cond$},
    legend pos=north west,
    ymode = log,
    xmode = log,
    log base y=10,
    log base x=10,
    ymin = 5e-17, ymax = 5e0,
    xmin = 5e-1, xmax = 5e15,
    max space between ticks=30,
    cycle list name = methodlist
    ]
    \addplot table[x=cond, y=naive] {accuracy_condA.dat};\addlegendentry{\algnaive}
    \addplot table[x=cond, y=sqrt_schur] {accuracy_condA.dat}; \addlegendentry{\algss}
    \addplot table[x=cond, y=cholesky_schur] {accuracy_condA.dat}; \addlegendentry{\algsc}
    \addplot table[x=cond, y=cholesky_schur_psd] {accuracy_condA.dat}; \addlegendentry{\algscp}
    \addplot table[x=cond, y=schur_schur] {accuracy_condA.dat}; \addlegendentry{\algscsc}   
    
    \addplot table[x=cond, y=condcond] {accuracy_condA.dat};    
    \addlegendentry{$u \cdot \hbox{cond}(\varphi; A, B)$}
    \end{axis}
    \alttext{Log-scale relative forward error against condition number, for five algorithms with f(x) = log⁡(x). The matrices have size n = 25. The matrix A has condition number between 10 and 10 to the power 15, while B has condition number 10. The two Cholesky--Schur algorithms have nearly identical errors within one order of magnitude of the unit roundoff. The square root--Schur and the naive algorithm have nearly identical error, which grows with the condition number of A. The double-Schur algorithm has the largest errors and follows the same trend.}
\end{tikzpicture}
    \hskip 40pt
    \begin{tikzpicture}[trim axis left, trim axis right]
    \begin{axis}[
    xlabel={$\cond$},
    legend pos=south east,
    ymode = log,
    xmode = log,
    log base y=10,
    log base x=10,
    ymin = 5e-17, ymax = 5e0,
    xmin = 5e-1, xmax = 5e15,
    max space between ticks=30,
    cycle list name = methodlist
    ]
    \addplot table[x=cond, y=naive] {accuracy_condAB.dat};\addlegendentry{\algnaive}
    \addplot table[x=cond, y=sqrt_schur] {accuracy_condAB.dat}; \addlegendentry{\algss}
    \addplot table[x=cond, y=cholesky_schur] {accuracy_condAB.dat}; \addlegendentry{\algsc}
    \addplot table[x=cond, y=cholesky_schur_psd] {accuracy_condAB.dat}; \addlegendentry{\algscp}
    \addplot table[x=cond, y=schur_schur] {accuracy_condAB.dat}; \addlegendentry{\algscsc}   
    
    \addplot table[x=cond, y=condcond] {accuracy_condAB.dat};    
    \addlegendentry{$u \cdot \hbox{cond}(\varphi; A, B)$}
    \legend{}
    \end{axis}
    \alttext{Log-scale relative forward error against condition number, for five algorithms with f(x) = log⁡(x). The matrices have size n = 25 and condition number between 10 and 10 to the power 15. The two Cholesky--Schur algorithms have nearly identical errors, which grow with the condition number. The remaining three algorithms have similar errors, which grow more steeply with the condition number.}
\end{tikzpicture}
    \caption{Relative forward error~\eqref{eq:rel-fwd-err} of the different implementations on random positive definite $n\times n$ matrices $A$ and $B$ with varying condition numbers for $n=25$ and $f(x) = \log(x)$, together with the quantity $u \cdot \hbox{cond}(\varphi; A, B)$. On the left, the condition number of $A$ grows ($\cond(A)=\mu$) while $B$ remains well conditioned ($\cond(B) = 10)$; on the right, the condition number of both matrices grows ($\cond(A)=\cond(B)=\mu$). The accuracy is measured using the relative forward error~\eqref{eq:rel-fwd-err}. The values reported are the median of 50 measurements across different matrices.}
    \label{fig:experiment2}
\end{figure}
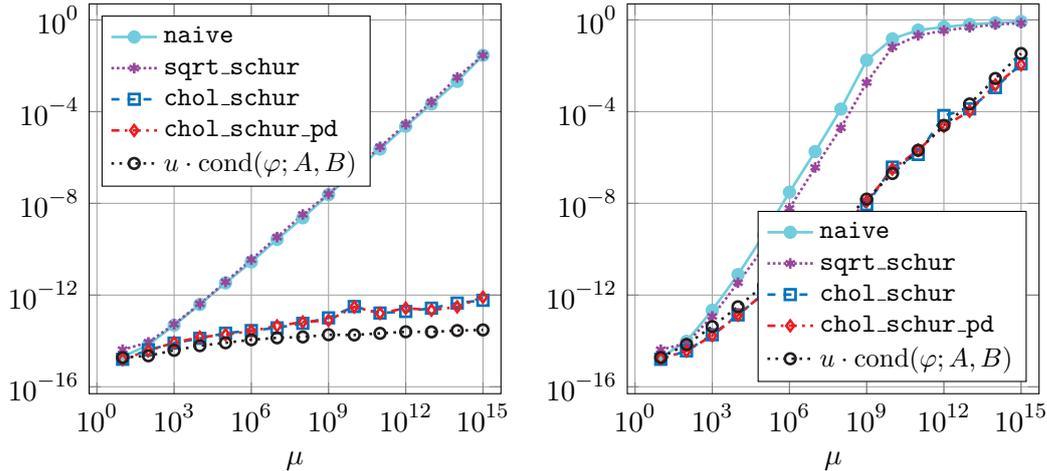

\Cref{fig:experiment2} reports the relative errors for matrices of size $n=25$ as the condition number of the matrix~$A$ increases. In the panel on the left, the matrix $B$ has condition number 10, while the condition number of $A$ grows as $10^i$ for $i=0,1,\ldots,15$.
In the panel on the right, the condition number of both $A$ and $B$ grows as $10^i$ for $i=0,1,\ldots,15$.
We compare the errors with the quantity $u\cdot \hbox{cond}(\varphi; A, B)$, which we compute by evaluating in floating-point arithmetic the definition~\eqref{eq:4}. For a backward stable algorithm, this quantity should have the same magnitude as the forward error.
Once again, the behavior of \algnaive{},
 \algss{}, and \algscsc{} is similar.
The errors of \algsc{} and
\algscp{} are almost indistinguishable. Their growth is more moderate, particularly when $B$ is well conditioned, and follows closely the trend of $u\cdot \hbox{cond}(\varphi; A, B)$, especially when $A$ and $B$ are both ill conditioned.

\Cref{fig:3} reports the timing of the algorithms as $n$ increases.
For all sizes considered, \algsc{} and \algscp{} are the fastest algorithms, with a remarkably similar runtime;
\algss{} and \algscsc{} perform similarly and are somewhat slower than \algsc{} and \algscp{}. Finally, \algnaive{} is one to two orders of magnitude slower than \algsc{} and \algscp{}.

\begin{figure}
    \centering
    \begin{tikzpicture}[trim axis left, trim axis right]
    \begin{axis}[
    xlabel={$n$},
    legend pos=south east,
    ymode = log,
    log base y=10,
    ymin = 5e-5, ymax = 5e0,
    xmin = -50, xmax = 1050,
    max space between ticks=30,
    cycle list name = methodlist
    ]
        \addplot table[x=n, y=naive] {time.dat};
        \addlegendentry{\algnaive}
        \addplot table[x=n, y=sqrt_schur] {time.dat}; \addlegendentry{\algss}
        \addplot table[x=n, y=cholesky_schur] {time.dat}; \addlegendentry{\algsc}
        \addplot table[x=n, y=cholesky_schur_psd] {time.dat}; \addlegendentry{\algscp}
       \addplot table[x=n, y=schur_schur] {time.dat}; \addlegendentry{\algscsc}       
    \end{axis}
    \alttext{Log-scale execution time (in seconds) versus matrix size n, for five algorithms with f(x) = log⁡(x). The matrices have size n between 50 and 1000. The two Cholesky--Schur algorithms are the fastest and have nearly identical timings. The square root--Schur and double-Schur algorithms follow a similar trend but are somewhat slower. The naive algorithm is the slowest, taking approximately one to two orders of magnitude longer than the Cholesky--Schur algorithms.}
\end{tikzpicture}
    \caption{Timings of the different implementations on $n\times n$ random positive definite matrices of increasing size $n$, with $f(x)=\log(x)$. The timings are expressed in seconds. The values reported are the medians of 50 measurements.}
    \label{fig:3}
\end{figure}
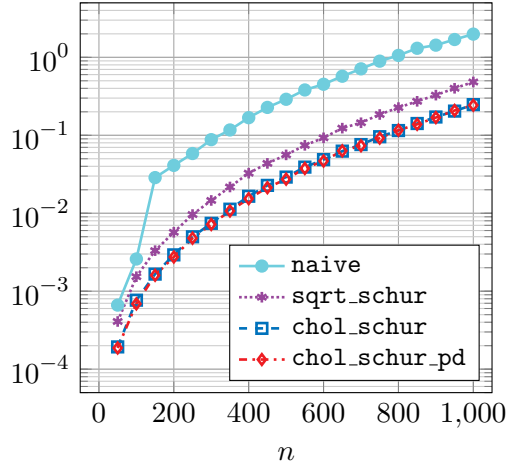

These results suggest that \cref{alg:2} (\algsc{}) and \cref{alg:3} (\algscp{}) are similar in terms of both accuracy and performance. In fact, the corresponding curves are almost indistinguishable and overlap in all figures. These two algorithms generally outperform the other alternatives.
In terms of accuracy, these results are consistent with the error analysis in \cref{sec:chsh-alg}, where we obtained a bound on the error of \cref{alg:2} that is smaller than that on the error of \cref{alg:1}. The experiments show something stronger: for moderate $n$, while the errors of \cref{alg:1,alg:ref} grow with the condition number of $A$, the errors of \cref{alg:2,alg:3} depend only on $\cond(B)$ and not on $\cond(A)$.

We repeated the experiment with $f(x) = x^{t}$, for $t \in (0, 1)$, and generating random matrices with different strategies.
In all cases we obtained similar results.

\section{Conclusions}
\label{sec:conclusions}
We studied the problem of computing $Af(A^{-1}B)$ when $A$ and $B$ are Hermitian matrices of moderate size with $A$ positive definite.
We analyzed its conditioning and considered five different algorithms. We considered the naive approach \algnaive{}, which first computes $A^{-1}B$ and then the matrix function $f(A^{-1}B)$, and \algss{}, which computes the square root of $A$ and uses the equivalent expression $A^{1/2}f(A^{-1/2}BA^{-1/2})A^{1/2}$. We introduced two new algorithms, \algsc{} and \algscp{}, which rely on the Cholesky factorization of $A${\color{revcol}, and we compared them with \algscsc{}, which relies on a double Schur decomposition}.

A rounding error analysis showed that the upper bound on the forward error of \algsc{} is better than that on the forward error of \algss{}.
Moreover, the computational complexity of the former is lower than that of the latter.

The results of this analysis are confirmed by the numerical experiments. The algorithms can be divided into two classes, one formed by \algnaive{}, \algss{}, and \algscsc{}, and the other comprising \algsc{} and \algscp{}. The algorithms in the latter class are generally faster and have a significantly smaller forward error. \algnaive{} is the slowest algorithm overall.

This work does not address the related problem of computing $Af(A^{-1}B)v$, where $A$ and $B$ are large and sparse and $v$ is a vector. The algorithms developed here would not be appropriate in that case, because the matrix $Af(A^{-1}B)$ is typically dense, too large to store in memory, and not needed in the application. These problems require the development of different techniques, such as those based on Krylov subspaces and on quadrature, which will be the subject of future work.

\section*{Appendix}
Here we explain how the asymptotic computational costs in \cref{tab:cost} can be obtained.
The costs for matrix multiplication, matrix inversion, Cholesky factorization, and Schur decomposition can be found in \cite[App.~C]{higham:book08}.

{\small
\begin{table}
    \centering
    \caption{
    Arithmetic cost of basic matrix operations for $n\times n$ matrices. $A$ and $B$ are general matrices, $T$ is lower or upper triangular, $L$ is lower triangular, $U$ is upper triangular,
    $H$ is Hermitian positive definite,
    and $S$ is either general or Hermitian.}
    \label{tab:cost}
    \begin{tabularx}{\linewidth}{Xcc}
    \toprule
    \multirow{2}[2]{*}{Operation} & \multicolumn{2}{c}{Number of flops}\\
    \cmidrule{2-3}
    & $S$ general & $S$ Hermitian \\
    \midrule
    Matrix multiplication $S=AB$, Matrix inversion $S=A^{-1}$     & $2n^3$ & $n^3$   \\
    Matrix multiplication $S=LU$     & $2n^3/3$ & $n^3/3$   \\
    Matrix multiplication $S=TB$ and  $S=T^{-1}B$ &  $n^3$ & $n^3/3$ \\
    Matrix multiplication $S=L_1L_2$ and $S=L_1^{-1}L_2$& $n^3/3$ & -- \\
    Cholesky factorization $H=R^*R$& -- & $n^3/3$\\
    Schur decomposition $S=QTQ^*$ & $25n^3$ & $9n^3$\\
    Matrix square root $H^{1/2}$ & -- & $10n^3$\\
    \bottomrule
    \end{tabularx}
\end{table}
}

The product $S=TB$, where $T$ is (lower) triangular and  $B$ is dense and unstructured, is equivalent to computing the action of $T$ on $n$ vectors of length $n$. Each product requires $n^2$ flops, so the overall cost is $n^3$ flops. Solving the linear equation $TX=B$ has the same cost, because applying a triangular matrix to a vector and solving a triangular system require the same number of operations.

If the product $S=TB$ is Hermitian, we can split $S$, $T$, and $B$ in four blocks and write
\[
\begin{bmatrix}
    s_1&s_2^*\\s_2&S_3
\end{bmatrix}=
\begin{bmatrix}
    t_1&0\\t_2&T_3
\end{bmatrix}
\begin{bmatrix}
    b_1&b_2^*\\b_3&B_4
\end{bmatrix},
\]
where $S_3,T_3$, and $B_4$ are square matrices of size $n-1$. Then we have
\[ s_1=t_1b_1,\quad s_2^*=t_1b_2^*,\quad S_3=T_3B_4+t_2b_2^*.
\]
Since $s_1$ and $s_2$ can be computed with $n$ multiplications, computing the first row and column of $S$ requires only $O(n)$ flops. The matrix
$S_3$ is Hermitian, and it suffices to compute its upper-triangular part.
Computing the upper triangular part of $t_2b_2^*$ requires $(n-1)(n-2)/2$ multiplications, and accumulating this on $T_3B_4$ requires \mbox{$(n-1)(n-2)/2$} additions. Computing $T_3B_4$ is equivalent to solving a problem of size $n-1$. Therefore, if we denote by $c_n$ the cost of computing $TB$, we obtain the recurrence $c_n=c_{n-1}+n^2-2n+2$, which for $c_1=1$ yields \mbox{$c_n=n^3/3+O(n^2)$} flops. The same argument applies to the solution of $TX=S$ when $S$ is Hermitian.

If we partition the factors of the product $L_1L_2$ in a similar manner, we obtain the same recurrence, which shows that this product also requires $n^3/3+O(n^2)$ flops. The same argument applies to the solution of the linear system $L_1 S = L_2$, where $S$ is Hermitian.

The cost of computing the matrix square root
of $A$ as $A^{1/2}=QD^{1/2}Q^*$, where $A=QDQ^*$ is the Schur decomposition of $A$, is the sum of the flop counts of the Schur decomposition and one matrix multiplication with Hermitian result.

\section*{Acknowledgments}
The authors are grateful to Marco Congedo, Fabio Durastante, and Volker Mehrmann for insightful discussions and for suggesting useful bibliographic references. Marco Congedo, in particular, pointed out Property \ref{it:p4} of \cref{thm:4}.

\bibliographystyle{plainnat}
\bibliography{biblio.bib}

\end{document}